\documentclass{amsart}
 \usepackage{amsmath,amssymb}
 \usepackage{times}
 
 \usepackage[all]{xy}

 \usepackage[plainpages=false,colorlinks,hyperindex,pdfpagemode=None,bookmarksopen,linkcolor=red,citecolor=blue,urlcolor=blue]{hyperref}
\usepackage{pdflscape}
\usepackage{multirow}

 \usepackage{cleveref}

\newtheorem{lemmaAM}{Lemma}[section]

\newtheorem{corollaryAM}[lemmaAM]{Corollary}

\newtheorem{theoremAM}[lemmaAM]{Theorem}

\newtheorem{propositionAM}[lemmaAM]{Proposition}

\theoremstyle{definition}
\newtheorem{definitionAM}[lemmaAM]{Definition}

\newtheorem{consAM}[lemmaAM]{Construction}
\newtheorem{exampleAM}[lemmaAM]{Example}

\newtheorem{questionAM}{Question}

 \swapnumbers

\theoremstyle{plain} 

\newtheorem{theorem}[subsection]{Theorem}

\newtheorem{lemma}[subsection]{Lemma}

\newtheorem{prop}[subsection]{Proposition}
\theoremstyle{definition}

\newtheorem*{remark*}{Remark}

\newtheorem*{theorem*}{Theorem}
\newtheorem*{defn*}{Definition}
\newtheorem*{lemma*}{Lemma}
\newtheorem*{corollary*}{Corollary}
\newtheorem*{conjecture*}{Conjecture}
 \newtheorem*{prop*}{Proposition}
 \newtheorem*{example*}{Example}

\newcommand{\cA}{\mathcal{A}}
\newcommand{\cC}{\mathcal{C}}
\newcommand{\cE}{\mathcal{E}}
\newcommand{\cF}{\mathcal{F}}
\newcommand{\cH}{\mathcal{H}}
\newcommand{\cK}{\mathcal{K}}
\newcommand{\cL}{\mathcal{L}}

\newcommand{\cO}{\mathcal{O}}
\newcommand{\cP}{\mathcal{P}}
\newcommand{\cS}{\mathcal{S}}

\newcommand{\bA}{\mathbf{A}}
\newcommand{\bC}{\mathbf{C}}
\newcommand{\bE}{\mathbf{E}}
\newcommand{\bF}{\mathbf{F}}

\newcommand{\bQ}{\mathbf{Q}}
\newcommand{\bZ}{\mathbf{Z}}

\newcommand{\Mod}{\mathrm{Mod}}
\newcommand{\Fun}{\mathrm{Fun}}
\newcommand{\Hom}{\mathrm{Hom}}
\newcommand{\Maps}{\mathrm{Maps}}
\newcommand{\LMod}{\mathrm{LMod}}
\newcommand{\Proj}{\mathrm{Proj}}
\newcommand{\Free}{\mathrm{Free}}
\newcommand{\Res}{\mathrm{Res}}
\newcommand{\Ind}{\mathrm{Ind}}
\newcommand{\KU}{\mathbf{K}}
\newcommand{\GL}{\mathrm{GL}}
\newcommand{\Ext}{\mathrm{Ext}}

\newcommand{\Cent}{\mathit{Z}}

\newcommand{\op}{\mathrm{op}}

\newcommand{\trace}{\mathrm{trace}}

\newcommand{\Vect}{\mathrm{Vect}}

\newcommand{\conj}{\mathrm{conj}}

\numberwithin{equation}{subsection}

\newcommand{\gl}{\mathrm{gl}}

\newcommand{\Gal}{\mathrm{Gal}} 
 
\newcommand{\KF}{K\!\cF}
\newcommand{\KP}{K\!\cP} 

\usepackage{fancyhdr}
\fancypagestyle{app}{
\fancyhead{}
\fancyfoot{}
\fancyhead[RO]{\scriptsize \thepage}
\fancyhead[LE]{\scriptsize \thepage}
\fancyhead[CO]{\scriptsize REPRESENTATIONS IN FINITE $\KU_p$-MODULES}
\fancyhead[CE]{\scriptsize AKHIL MATHEW}

}

\title{Representations of finite groups on modules over $K$-theory}
\author[David Treumann]{David Treumann, with an appendix by Akhil Mathew}

\begin{document}
 
 \maketitle
 
 \begin{abstract} 
Let $G$ be a finite group, and let $\KU_p$ denote the completion at $p$ of the complex $K$-theory spectrum.  $\KU_p$ is a commutative ring spectrum that in some ways is very similar to the usual ring $\bZ_p$ of $p$-adic integers.  We discuss $G$-actions on $\KU_p$-modules, and propose to study them by analogy with the classical theory of modular representations of $G$.  
\end{abstract}

\tableofcontents

\section{Introduction}

Let $\KU$ denote the complex $K$-theory spectrum.  It is a commutative ring spectrum and we may consider modules over it.  In this paper, we tweak some definitions in representation theory, replacing abelian groups by $\KU$-modules, to see what happens.  This is a reasonable experiment:

\begin{prop*}
The category of $\KU$-modules and homotopy classes of maps between them is equivalent to the $2$-periodic derived category of abelian groups.  The full subcategory of ``even'' $\KU$-modules --- those whose odd homotopy groups vanish --- is equivalent to the category of abelian groups and group homomorphisms between them.
\end{prop*}

A richer structure appears if we do not pass to homotopy classes: $\KU$-modules form an $\infty$-category.  That is, given $\KU$-modules $M$ and $N$ there is a space $\Maps(M,N)$ of homomorphisms between them.  The proposition describes the structure of the set of connected components of $\Maps(M,N)$, denoted $[M,N]$.  Each of these components is an infinite-dimensional space --- a well-studied one in algebraic topology, closely related to the stable unitary group $\mathrm{U} := \bigcup_n \mathrm{U}(n)$ and its classifying space $B \mathrm{U}$.  

We will study representations of a finite group $G$ on $\KU$-modules --- let us call them $\KU[G]$-modules.  The $\infty$-categorical structure plays a big role: by definition, a $G$-action on a $\KU$-module $M$ is a basepoint-preserving map $BG \to B\mathrm{Aut}(M)$.  The category of such things, and homotopy classes of maps between them, forms a triangulated category very different from the derived category of $\bZ[G]$-modules.  Our goal is to compare $\bZ[G]$-modules to $\KU[G]$-modules, and perhaps to take advantage of the differences.

\subsection{Motivation and Atiyah's theorem}
\label{intro:Atiyah}
I am looking for a representation-theoretic format to take advantage of the fact that, for the classifying space of a finite group, it is easier to compute $K$-theory than cohomology.  The determination of $H^*(BG)$ can be cast as a representation-theoretic computation --- it is the ring of (homologically shifted) endomorphisms of the trivial $G$-module.  Indeed this ring plays an important role in representation theory, giving a notion of support of $G$-modules, e.g. \cite{AlperinEvens,AvruninScott}.  But it is difficult to compute.

Meanwhile Atiyah has given \cite{Atiyah} a beautiful description of the $K$-theory of $BG$:
\begin{equation}
\label{eq:atiyah}
\KU^0(BG) = R(G)_{I} \qquad \KU^1(BG) = 0
\end{equation}
where $R(G)$ denotes the Grothendieck ring of complex representations of $G$, and $R(G)_{I}$ denotes its completion at the augmentation ideal.  In particular \cite[Lem. 8.3]{Atiyah},  $\KU^*(BG)$ is torsion-free and contains no nilpotent elements.

Thus, the endomorphism ring of the trivial $\KU[G]$-module is easier to compute than the endomorphism ring of the trivial $\bZ[G]$-module.  Possibly other computations for $\KU[G]$-modules are also easier.  But let us show that the theory is neither trivial, nor unrelated to classical modular representation theory.

\subsection{Blocks}
\label{intro:blocks}
An abelian group can be reduced mod a prime $p$, or completed at a prime $p$, yielding $\bF_p$- and $\bZ_p$-modules.  These operations break the category of $G$-modules into ``blocks.''  Let us quickly give a definition of block here: consider the graph whose vertices are the indecomposable representations of $G$ on $\bZ_p$-modules, and whose edges are nonzero maps.  A block is a connected component of this graph.  One gets the same set of blocks (at least, the components of the graph are canonically identified) if one uses $\bF_p$-modules in place of $\bZ_p$-modules, or if one considers complexes of modules up to shift.   

There is a good $K$-theoretic analog of $\bZ_p$, called $p$-complete $K$-theory, which we denote by $\KU_p$.  There is also an analog of $\bF_p$, called mod $p$ $K$-theory, but the analogy is more problematic: for instance it is not commutative when regarded as a ring spectrum.  So let us focus  on $\KU_p$.  Considering $\KU_p[G]$-modules, one gets a similar notion of blocks.  It is easy to deduce the following:

\begin{theorem*}
The blocks of the categories of $\KU_p[G]$-modules and of $\bZ_p[G]$-modules are in natural one-to-one correspondence.  
\end{theorem*}

We give a proof in \S\ref{subsec:blocks} based on the  observation that the category of $\KU_p[G]$-modules has a big chunk familiar from ordinary representation theory --- the homotopy category of \emph{projective} $\KU_p[G]$-modules is equivalent to the category of projective $\bZ_p[G]$-modules.  The proof does not appeal to any special feature of $K$-theory, except that its coefficient groups are a complete local domain that vanishes in odd degrees.

\subsection{Kuhn's theorem}
\label{intro:Kuhn}
Kuhn gives a variant of Atiyah's computation \eqref{eq:atiyah}.  If we replace $\KU$ with $\KU_p$, then $\KU_p^*(BG)$ is naturally isomorphic to the quotient of $\bZ_p \otimes R(G)$ by the ideal of virtual representations that vanish when restricted to a $p$-Sylow subgroup of $G$.  (Let us call this the ``Kuhn ideal'' of $\bZ_p \otimes R(G)$.)  This is a free $\bZ_p$-module of rank equal to the number of conjugacy classes of $p$-power order in $G$.  Kuhn proves this in \cite{Kuhn} for mod $p$ $K$-theory, the $p$-adic version important for us is proved in \cite[Prop. 9.7]{Strickland}.

It will be useful to have a name for the following variation on the notion of the Kuhn ideal: if $\cE$ is a virtual $G$-equivariant vector bundle on a $G$-space $X$, we will say that $\cE$ is ``Kuhn-trivial'' if each stalk $\cE_x$, regarded as a virtual representation of the isotropy group $G_x$, belongs to the Kuhn ideal of $R(G_x) \otimes \bZ_p$.

\subsection{Roots of unity}
In representation theory in any characteristic it is usually desirable to work with coefficients that contain sufficiently many roots of unity.  If $\bF_q$ is an extension of $\bF_p$ and $\bZ_q \supset \bZ_p$ are the Witt vectors of $\bF_q$, then $\bF_q[G]$ and $\bZ_q[G]$ have the same blocks.

The ring $\bZ_q$ is obtained by adding $(q-1)$th roots of unity to $\bZ_p$.  An analogous construction yields an extension $\KU_q$ over $\KU_p$.  (It is the ``first Morava $E$-theory'' attached to $\bF_q$.)  One can replace $\KU_p$ with $\KU_q$, and analogs of the results above continue to hold: the homotopy category of $\KU_q$-modules is the $2$-periodic derived category of $\bZ_q$-modules, and the blocks of $\KU_q[G]$-modules are in bijection with the blocks of $\bZ_q[G]$-modules (and with the blocks of $\bF_q[G]$-modules). 

It is important that $\bZ_q$ contains no $p$th roots of unity: a standard argument using the power operations of \cite{Hopkins} shows that it is impossible to adjoin them to $\KU_p$.  In fact there is a sense in which $\KU_p$ is already the maximal $p$-cyclotomic extension of the $K(1)$-local sphere spectrum --- this is one formulation of the Adams conjecture.

\subsection{New symmetries of modular representation theory}
\label{intro:newsymmetries}
We will write $\LMod(\KU_q[G])$ for the stable $\infty$-category of $\KU_q[G]$-modules, and $\LMod(\KU_q[G])^{\mathit{ft}}$ for its full subcategory of modules of ``finite type'', i.e. $\KU_q[G]$-modules whose homotopy groups are finitely generated over $\bZ_q$.

We prove the following in \S\ref{subsec:Koszul}:

 \begin{theorem*}
 Let $G$ be a finite commutative $p$-group, and let $G^{\sharp} = \Hom(G,\bC^*)$ denote its Pontrjagin dual.  Then there is a canonical equivalence
 \begin{equation}
 \label{eq:K-Koszul}
\LMod(\KU_q[G])^{\mathit{ft}} \cong \LMod(\KU_q[G^\sharp])^{\mathit{ft}}
 \end{equation}
 \end{theorem*}
 Of course, $G$ and $G^\sharp$ are noncanonically isomorphic, but the equivalence we have in mind is not induced by any such isomorphism.  Another way to put this is that (once an isomorphism $G \cong G^{\sharp}$ is fixed), the category of $\KU_q[G]$-modules has an exotic automorphism that the category of $\bZ_q[G]$-modules does not have.

A sample computation in \S\ref{subsec:GKM}, which is formally similar to one of the results of \cite{GKM}, illustrates the equivalence.  If $X$ is a space with a $G$-action, then $\KU_q[X]:= \KU_q \wedge (\Sigma^\infty X_+)$ is a $\KU_q[G]$-module in a natural way.  We call it a ``transformation module,'' or a ``permutation module'' if $X$ is a finite set.  If $\KU_q[X]$ has finite type, then the image of $\KU_q[X]$ under \eqref{eq:K-Koszul} is $L_{\hat{p}}\KU_q[X_{hG}]$, the $p$-completion of the $K$-theory of the Borel construction.  The Pontrjagin dual acts by twisting $G$-equivariant vector bundles on $X$ by homomorphisms $G \to \bC^*$.

\subsection{Borel versus Bredon}
\label{intro:borelbredon}

A $G$-action on a spectrum should represent a $G$-equivariant cohomology theory.  There are two points of view on equivariant cohomology, and two corresponding points of view of $G$-spectra.  One point of view is Borel's: if $X$ is a $G$-space and $\bA$ is an extraordinary cohomology theory, then the $G$-equivariant $\bA$-cohomology of $X$ is the $\bA$-cohomology of the Borel construction $X_{hG}$.  From this point of view, a $G$-action on a spectrum (or a $\KU$-module spectrum, or another kind of object from algebraic topology) is a functor from $BG$ to the $\infty$-category of spectra.

The second point of view is much richer, and culminates in a category of spectra that represent $R\mathrm{O}(G)$-graded cohomology theories.  These are often called ``genuine $G$-spectra.''  An early contribution along these lines is Bredon's \cite{Bredon}.  The $K$-theory spectrum (but not every spectrum) has an avatar in the Bredon world, let us denote it by $\KU_G$, its $p$-completion by $\KU_{G,p}$, and its unramified extensions by $\KU_{G,q}$.

We have chosen to work in the simpler Borel world, and not only because it requires less background.  For instance, in the Bredon world the analog of the theorem of \S\ref{intro:blocks} is false --- according to the theory of Segal-tom Dieck splitting and the results of \cite{Dress1969}, each $p$-perfect subgroup of $G$ contributes its own blocks to the Bredon-style category.  I suspect that these other blocks are not really new and interesting.  For example, I make the following conjecture, which is closely related to a question raised by Mathew in the appendix:

\begin{conjecture*}
There is an equivalence of stable $\infty$-categories
\[
\LMod(\KU_{G,q})^\omega \cong \bigoplus_c \LMod(\KU_q[\Cent_G(c)])^{\mathit{ft}}
\]
where 
the left-hand side denotes the compact objects in the category of $\KU_{G,q}$-modules in genuine equivariant spectra, 
the sum on the right-hand side is over representatives for conjugacy classes of elements $c \in G$ of order prime to $p$, and the symbols $\Cent_G(c)$ denote centralizers.
\end{conjecture*}

On the other hand, $\Fun(BG,\Mod(\KU))$ is not as good an approximation for the global (i.e. not completed at a prime) category of $\KU_G$-modules in genuine spectra.  The latter category is very interesting but not as closely analogous to structures of classical modular representation theory.  I hope to return to this elsewhere.

\subsection{$\KU_q$-representation theory as a smooth deformation of $\bZ_q$-representation theory}
There are two tools from algebraic topology which give a striking interpretation: the ``noncommutative smoothness'' of modular representation theory over $\KU$.  The two tools are the Adams spectral sequence, and the vanishing of the Tate construction \cite{GreenleesSadofsky}.  The relevance of the Greenlees-Sadofsky result to the noncommutative smoothness of group algebras and generalizations of them is the subject of a recent paper by Hopkins and Lurie \cite{HopkinsLurie}.

If $M$ is a $\KU_q[G]$-module, let $\pi_0 M,\pi_1M$ denote the homotopy groups of $M$ with their actions of $G$.  There are many nonzero morphisms on $\KU_q[G]$-modules which induce zero on homotopy groups --- analyzing this phenomenon carefully leads to the Adams filtration
\[
\cdots \subset [M,N]^{\geq 2} \subset [M,N]^{\geq 1} \subset [M,N]^{\geq 0} = [M,N]
\]
where $f \in [M,N]^{\geq k}$ if it can be factored as a composite of $k$ maps, each inducing zero on homotopy groups.  In good cases the associated graded of this filtration is the target of a spectral sequence, whose $E_2$ page is the $\Ext$-group between $\pi_* M$ and $\pi_* N$ taken in the category of $\bZ/2$-graded $\bZ_q[G]$-modules.  So we informally regard $\KU_q[G]$-module theory as a ``deformation'' of $\bZ_q[G]$-module theory.  We now explain a sense in which this deformation is \emph{smooth}.

Following \cite{KoSo}, let us call an associative algebra spectrum smooth if its diagonal bimodule is compact in the $\infty$-category of bimodules --- this means that the covariant functor represented by the diagonal bimodule preserves direct sums (infinite ones).  For the group algebra of a finite group $G$, Frobenius reciprocity shows this notion to be equivalent to the assertion that the trivial module is compact.  Over $\bZ_q$, this fails in every interesting case, whenever the residue characteristic divides the order of $G$.

The $K$-theoretic group ring does better, but to obtain a correct statement we have to replace $\LMod(\KU_q[G])$ with another category called $L_{\hat{p}}\LMod(\KU_q[G])$ --- the $p$-completion of the category of $\KU_q[G]$-modules.  This difference between these categories is somewhat faint --- they are related by adjoint functors that become equivalences when restricted to finite type modules.  Still, infinite direct sums are different in the two categories.\footnote{We are not holding $\KU_q$-coefficients to a lower standard of smoothness than $\bZ_q$-coefficients: the trivial $\bZ_q[G]$-module is not compact in either $\LMod(\bZ_q[G])$ or $L_{\hat{p}}\LMod(\bZ_q[G])$}

The space of maps out of the trivial module to another $\KU_q[G]$-module $M$ is closely related to the homotopy fixed point spectrum $M^{hG}$, and the assertion that the trivial $\KU_q[G]$-module is compact would be implied by the assertion that $M \mapsto M^{hG}$ preserves direct sums.  This is not true in $\LMod(\KU_q[G])$, but it is true in $L_{\hat{p}}\LMod(\KU_q[G])$ --- this is a special case of ``the vanishing of the Tate construction,'' a miracle of the $K(1)$-local world that identifies $M^{hG}$ with a functor $L_{K(1)}M_{hG}$ that more obviously preserves direct sums.

\subsection{$p$-permutation modules}
If $X$ is a $G$-set and $\bZ_q[X]$ is the corresponding permutation $G$-module, the indecomposable direct summands of $\bZ_q[X]$ are known as \emph{$p$-permutation modules}, \cite{Broue}.  The same definition makes sense with $\bZ_q[X]$ replaced by $\KU_q[X]$.  In \S\ref{sec:four}, we will prove

\begin{theorem*}
The functor $M \mapsto \pi_0 M$ induces a bijection between isomorphism classes of indecomposable $p$-permutation $\KU_q[G]$-modules and indecomposable $p$-permutation $\bZ_q[G]$-modules.
\end{theorem*}

Projective modules are examples of $p$-permutation modules, and as we have mentioned already in \S\ref{intro:blocks} for projective modules the functor of the theorem is an equivalence of homotopy categories.

The endomorphism ring of $\KU_q[X]$ can be computed using \S\ref{intro:Kuhn} as a quotient (and $p$-completion) of the ring $K_G(X \times X)$ of virtual $G$-equivariant vector bundles on $X$, with its convolution structure.  This ring has been studied by Lusztig \cite{Lusztig1}. Thus the theorem is reduced to a statement of pure algebra fairly easily.  We have therefore tried to keep \S\ref{sec:four} self-contained.  

We actually prove a more general statement, avoiding the Kuhn quotient.  We study the Karoubi completion of the category whose objects are finite $G$-sets and whose morphisms from $X$ to $Y$ are given by $K_G(X \times Y) \otimes \bZ_q$.   Using some results of Bonnaf\'e about the blocks of $R(G)\otimes \bZ_q$, we show that the isomorphism classes of indecomposable objects of this category are in bijection with conjugacy classes of pairs $(g,M)$ where $g \in G$ has order prime to $p$ and $M$ is a usual $p$-permutation module for $\Cent_G(g)$.  The conjecture of \S\ref{intro:borelbredon} is partly motivated by this decomposition, as well as by the result of Mathew in the appendix.

\subsection{The character ring}
\label{intro:19}

Let us denote by $\Cent(\KU_q[G])$ the ring of natural endomorphisms of the identity functor on the homotopy category of $\LMod(\KU_q[G])$.  If we let $G_{\conj, hG}$ denote the Borel construction of $G$ with its conjugation action, then a very general result of \cite{BZFN} identifies $\Cent(\KU_q[G])$ with $\KU_q^0(G_{\conj, hG})$ --- by \S\ref{intro:Atiyah} and \S\ref{intro:Kuhn}, this is the group of conjugation-equivariant virtual vector bundles on $G$, divided by the Kuhn-trivial vector bundles.  It is a free $\bZ_q$-module whose rank is the cardinality of the set of conjugacy classes of commuting pairs $(u,g)$, where $u$ has $p$-power order.

The precise structure of $\Cent(\KU_q[G])$ --- additively and multiplicatively --- is more subtle.  For example if $G$ is a commutative $p$-group, there is a canonical isomorphism of rings $\Cent(\KU_q[G]) \cong \bZ_q[G \times G^{\sharp}]$, reflecting the symmetry of \S\ref{intro:newsymmetries}.  In general the structure of $\Cent(\KU_q[G])$ can be deduced from results of \cite{Lusztig1} and \cite{Bonnafe2}.  Before describing it, we introduce some notation.  

Let us say that an ``enriched conjugacy class''  of order $n$ in $G$ is a conjugacy class $C \subset G$ of order $n$, together with an irreducible conjugation-equivariant vector bundle $\cL$ over $C$.  Put another way, an enriched conjugacy class is a conjugacy class of pairs $(g,L)$, where $g \in G$ and $L$ is an irreducible complex representation of $\Cent_G(g)$.  There is an action of $\hat{\bZ}^* \times \hat{\bZ}^*$ on enriched conjugacy classes --- $(\gamma_1,\gamma_2) \in \hat{\bZ}^* \times \hat{\bZ}^*$ acts by raising $g$ to the $\gamma_1$th power and by sending $L$ to its Galois conjugate by $\gamma_2$ (thinking of $\gamma_2$ as belonging to $\Gal(\bQ(\mu_{\infty})/\bQ)$).

\begin{theorem*}
The commutative ring $\Cent(\KU_q[G])$ is reduced, free of finite rank over $\bZ_q$, and of Krull dimension one.  The maximal prime ideals of $\Cent(\KU_q[G])$ are in one-to-one correspondence with blocks of $\bZ_q[G]$.  If $q$ is sufficiently large, the minimal prime ideals of $\Cent(\KU_q[G])$ are naturally in bijection with $\bZ_p^*$-orbits of enriched conjugacy classes of $p$-power order, where $\bZ_p^* \subset \hat{\bZ}^* \times \hat{\bZ}^*$ is embedded diagonally. 
\end{theorem*}

If $u \in G$ has $p$-power order, there is a surjective map from the blocks of $\Cent_G(u)$ to the blocks of $G$ whose defect group meets the conjugacy class of $u$.  This map is a special case of the Brauer correspondence, the case relevant in his ``second main theorem'' \cite{Brauer,Dade}.  This map matches the ``specialization'' map from minimal to maximal primes in $\Cent(\KU_q[G])$.  More precisely, each pair $(u,L)$ determines a block $b$ of $\Cent_G(u)$ --- the block containing all the $\bZ_q[G]$-lattices in $L$ --- and the block $B$ corresponding to the specialization of $(u,L)$ is the Brauer correspondent of $b$.
 
\subsection{Support}
\label{intro:support}
We have mentioned in \S\ref{intro:Atiyah} the notion of support for $G$-modules in classical representation theory.  Let us now discuss an analog in $K$-theoretic representation theory --- it is rather an analog of a refined version of support considered by Linckelmann \cite{Linckelmann}.

As $\Cent(\KU_q[G])$ is the endomorphism ring of the identity functor, for any $M \in \LMod(\KU_q[G])$ there is a ring map
$
\Cent(\KU_q[G]) \to [M,M]
$.
We define the support of $M$ to be the set of associated prime ideals of the kernel of this map.

\begin{example*}
Let $X$ be a finite $G$-set and $(u,L)$ an enriched conjugacy class of $p$-power order.  The minimal prime ideal corresponding to $(u,L)$ is in the support of the permutation module $\KU_q[X]$ if and only if $L$ appears in the $\bC[\Cent_G(u)]$-permutation module $\bC[X^u]$.  
\end{example*}

The example shows that, in particular, whenever $u$ acts nontrivially on $L$, the miminal prime ideal corresponding to $(u,L)$ supports no permutation $\KU_q[G]$-modules.  It would follow from the Conjecture of \S\ref{intro:borelbredon} that there is no finite-type $\KU_q[G]$-module at all whose support contains such a $(u,L)$ --- if $G$ is cyclic of order $p$ this follows unconditionally from Mathew's result in the appendix.  This is a curious phenomenon, uninhabited regions in the spectrum of the Hochschild cohomology ring.

\subsection{Questions}

\subsubsection{Realization of modules}
\label{intro:q1}
If $\overline{M}$ is a $\bZ_q[G]$-module, when can we find a $\KU_q[G]$-module $M$ with $\pi_0 M = \overline{M}$ and $\pi_1 M = 0$?  We call such an $M$ an even realization of $\overline{M}$.  In \S\ref{sec:carlsson}, we adapt Carlsson's counterexample to the equivariant Moore space problem \cite{Carlsson} to show that this is not always possible.  

But we have already discussed a positive solution for projective modules and more general $p$-permutation modules, and in \S\ref{subsec:one-dimensional} we consider the case where $\overline{M}$ free of rank one over $\bZ_q$.  For other natural classes of 
modules, I am not sure what to expect:
\begin{enumerate}
\item Can every irreducible $\bF_q[G]$-modules be evenly realized over $\KU_q[G]$?  
\item Does every $\bQ_q[G]$-module have a $\bZ_q[G]$-lattice that can be evenly realized over $\KU_q[G]$?
\end{enumerate}

\subsubsection{Equivalences between blocks}
\label{intro:q2}
Given finite groups $G,G'$, a block $b$ of $G$ and a block $b'$ of $G'$, does the existence of a derived equivalence $D^b(\bF_q[G]b) \cong D^b(\bF_q[G']b')$ imply that of an equivalence $\LMod(\KU_q[G]b)^{\mathit{ft}} \cong \LMod(\KU_q[G']b')^{\mathit{ft}}$, or conversely?  The result discussed in \S\ref{intro:newsymmetries} shows that the groups of self-equivalences  of $\bF_q[G]b$ and of $\KU_q[G]b$ can be different.

In $\bZ_q$- and $\bF_q$-representation theory, an interesting class of derived equivalences between blocks are those that can be given by two-term complexes of $(G,G')$-bimodules of the form $P \to M$, where $P$ is a projective bimodule and $M$ is a $p$-permutation bimodule.  For instance, such complexes arise in Rouquier's solution to the Brou\'e conjecture for blocks of cyclic defect \cite[\S 10.3]{Rouquier}, or in Rickard's solution to the Brou\'e conjecture for the principal block at $p = 2$ of the alternating group $\mathrm{A}_5$ \cite[\S 3]{Rickard}.  These bimodules have very natural realizations (not even ones) in the $K$-theoretic world --- do they still give equivalences?

\subsubsection{Bredon version of \S\ref{intro:newsymmetries}} In the notation of \S\ref{intro:borelbredon}, it is likely that $\LMod(\KU_G) \cong \LMod(\KU_{G^\sharp})$ for any finite commutative group $G$.  Can the construction in the Example of \S\ref{subsec:one-dimensional}, which gives the bimodule realizing the equivalence of \S\ref{intro:newsymmetries}, be upgraded to a ``genuine'' $G$-spectrum?

\subsubsection{Brauer functor}  What should be the theory of the Brauer functor, in the $\KU_q$-story?  

In classical modular representation theory, given a $p$-subgroup $Q \subset G$, the Brauer functor is a functor from $p$-permutation modules for $\bF_q[G]$ to those for $\bF_q[N_G(Q)]$, which is used constantly.  It is a Smith-theoretic construction, sending a summand of the permutation representation for the $G$-set $X$ to a summand of the permutation representation of the $N_G(Q)$-set $X^Q$.  

For $\KU_q[G]$-modules, it seems that something like this is available only when $Q$ is cyclic, in which case $X \mapsto X^Q$ extends to a functor from $p$-permutation modules for $\KU_q[G]$ to modules for $\bQ_q(\zeta_{p^e})[\Cent_G(Q)]$ --- here $p^e$ is the order of $Q$ and $\zeta_{p^e}$ is a $p^e$th root of unity.  The mismatch of coefficients here (is it related to the ``blueshift'' of \cite{HoveySadofsky}?) is partly disappointing --- it makes it difficult to iterate the construction.  But the $\bF_q$-Brauer functor does not lift to characteristic zero, can one make use of this?

\subsection{Acknowledgments}
I learned a lot of $K$-theory from Dustin Clausen and Tyler Lawson, and I benefited from correspondence with Raphael Rouquier about some of these ideas.  I am also grateful to Akhil Mathew for the appendix, and for many comments and corrections to the paper. I was supported by NSF-DMS-1206520 and a Sloan Foundation Fellowship, and some of this paper was written at the MSRI in Berkeley.

\subsection{Notation}
Write $\Cent_G(g)$ for the centralizer of an element $g$ in a group $G$.  We let $\mathrm{ad}_g$ denote the conjugation action of $g$: $\mathrm{ad}_g(x) := gxg^{-1}$.

A spectrum $\mathbf{E}$ determines an extraordinary cohomology theory.  The $i$th $\mathbf{E}$-cohomology of $X$ will be denoted in two ways (depending on what is less of an eyesore) --- either as $\mathbf{E}^i(X)$ or as $H^i(X;\mathbf{E})$.

If $X$ and $Y$ are objects of an $\infty$-category write $\Maps(X,Y)$ for the space of maps between them, and $[X,Y]$ for the set of connected components of $\Maps(X,Y)$.

``Ring spectrum'' or ``associative ring spectrum'' means $\mathrm{E}_1$-algebra object in spectra.  Write $\LMod(\bA)$ for the category of left modules over an associative ring spectrum.  ``Commutative ring spectrum'' means $\mathrm{E}_\infty$-algebra object in spectra.  Write $\mathrm{Mod}(\bA) = \LMod(\bA)$ for the symmetric monoidal category of (left) modules over a commutative ring spectrum.

We write $\Sigma$ for the suspension functor in a stable $\infty$-category.  Thus, $\Sigma M = M[1]$ in the usual cohomological triangulated category conventions.

If $X$ is a space $X_+$ denotes the disjoint union of that space and a point.  $\Sigma^{\infty} X_+$ denotes the suspension spectrum of $X$.  

If $X$ is a space with a $G$-action, let $X_{hG}$ denote the Borel construction, i.e. $X_{hG} = (X \times EG)/G$ where $EG$ is contractible with free $G$-action.  We write $X^{hG}$ for the space of homotopy fixed points.  We also apply the sub- and super-script $hG$ to spectra with $G$-actions, i.e. to functors from $BG$ to the $\infty$-category of spectra, in which case
\[
\bE_{hG} := \varinjlim_{BG} \bE, \qquad \qquad \bE^{hG} := \varprojlim_{BG} \bE
\]
the direct and inverse limits taken in the $\infty$-categorical sense.

\section{Representations over a good coefficient algebra}

We will call a commutative ring spectrum $\bA$ a \emph{good coefficient algebra} if $\pi_0\bA$ is a principal ideal domain, and there is a class $\beta \in \pi_2\bA$ such that $\pi_*\bA \cong (\pi_0\bA)[\beta,\beta^{-1}]$.  In particular, the homotopy groups of a good coefficient algebra vanish in odd degrees.  The $K$-theory spectrum and its completions are the examples we are interested in, but in this section we study those features of $\KU[G]$-modules that can be deduced from this ``good coefficients'' property.

If $\bA$ is a good coefficient algebra the $1$-categorical structure of $\Mod(\bA)$ is very simple:

\begin{prop}
\label{thm:coefficients}
Let $\bA$ be a good coefficient algebra.  Every object $M \in \Mod(\bA)$ is isomorphic to a direct sum $M_0 \oplus \Sigma M_1$, where $M_0$ and $M_1$ are ``even'', i.e. have vanishing odd homotopy groups.  If $M$ and $N$ are even, there are natural isomorphisms
\begin{enumerate}
\item $[M,N] \stackrel{\sim}{\to} \Hom(\pi_0(M),\pi_0(N))$
\item $[M,\Sigma N] \stackrel{\sim}{\to} \Ext^1(\pi_0(M),\pi_0(N))$
\end{enumerate}
The first isomorphism sends the homotopy class of a map $f:M \to N$ the induced map on $\pi_0$, and the second isomorphism sends the homotopy class of a map $f:M \to \Sigma N$ to the class in $\Ext^1$ of the extension
\[
0 \to \pi_0(N) \to \pi_0(\Sigma^{-1} \mathrm{Cone}(f)) \to \pi_0(M) \to 0
\]
\end{prop}

Put another way, the homotopy category of $\Mod(\bA)$ is equivalent to a category whose objects are $\bZ/2$-graded $\pi_0 \bA$-modules, and whose morphisms from $(M_0,M_1)$ to $(N_0,N_1)$ are given by $2 \times 2$ matrices whose diagonal entries are $\Hom(M_0,N_0)$ and $\Hom(M_1,N_1)$, and whose off-diagonal entries are taken from $\Ext^1(M_0,N_1)$ and $\Ext^1(M_1,N_0)$.

\begin{proof}
We use the fact that $\pi_i M \cong [\Sigma^i \bA,M]$, and that over a principal ideal domain a submodule of a free module is free.

Let us call an $\bA$-module free if $\pi_* M$ is free as a $\pi_* \bA$-module.  If $M$ is free, a $\pi_*\bA$-basis for $\pi_* M$ determines a map $\bA^{\oplus i} \oplus \Sigma \bA^{\oplus j} \to M$.  From the identification $\pi_i M \cong [\Sigma^i \bA,M]$ we deduce this map is a isomorphism --- the map from the $\bA^{\oplus i}$ summand induces an isomorphism on $\pi_0$ and the zero map on $\pi_1$, while the map from $\Sigma \bA^{\oplus j} \to M$ induces an isomorphism on $\pi_1$ and the zero map on $\pi_0$.  It follows that if $M$ and $N$ are free, then $\pi_*$ gives an isomorphism from $[M,N]$ to $\Hom(\pi_0 M,\pi_0 N) \oplus \Hom(\pi_1 M,\pi_1 N)$.  

Suppose now that $M \in \LMod(\bA)$ is any object and let us show that $M \cong M_0 \oplus \Sigma M_1$, where $M_0$ and $M_1$ are even.  A set of generators of $\pi_* M$ determines a map $f:A^{\oplus i} \oplus \Sigma \bA^{\oplus j} \to M$ that induces a surjection on $\pi_0 M$ and $\pi_1 M$. From the long exact sequence of homotopy groups and the fact that submodules of free modules are free, it follows that the fiber of $f$ is also free, say $\bA^{\oplus i'} \oplus \Sigma \bA^{\oplus j'}$.  For $M_0$ we take the cone on the map $\bA^{\oplus i'} \to \bA^{\oplus i}$ and for $M_1$ the cone on the map $\bA^{\oplus j'} \to \bA^{\oplus j}$.

It remains to show that (1) and (2) hold.  Suppose $M,N \in \LMod(\bA)$ are even.  We may write $M$ as the cone on a map of even, free $\bA$-modules $F_1 \to F_0$ that induces an injection on homotopy groups.  From the discussion so far it is clear that $\pi_0 F_1 \to \pi_0 F_0$ is a free resolution of $\pi_0 M$, that $[F_1,\Sigma^{-1} N]$ and $[F_0,\Sigma N]$ vanish, and that $[F_i,N] \cong [\pi_0 F_i,\pi_0 N]$.  The isomorphisms (1) and (2) can now be deduced from the long exact sequence
\[
[F_1,\Sigma^{-1} N] \to [M,N] \to [F_0,N] \to [F_1,N] \to [M,\Sigma N] \to [F_0,\Sigma N]
\]
induced by the exact triangle $F_1 \to F_0 \to M \to$.
\end{proof}

\subsection{Roots of unity}
Let us say a good coefficient algebra $\bA$ is ``$p$-adic'' if $\pi_0 \bA$ is a complete discrete valuation ring of residue characteristic $p$.  (We have in mind the mixed characteristic case, but will not make use of this property.)  If $\bA$ is a $p$-adic good coefficient algebra, then for any prime-to-$p$ cyclotomic extension $\cO$ of $\pi_0 \bA$, there is a unique commutative $\bA$-algebra spectrum $\cA$ such that $\cA$ is a good coefficient algebra and $\pi_0(\cA) \cong \cO$.

\begin{proof}
This is a special case of a standard construction, see for instance \cite[\S 2]{BakerRichter} or \cite[Th. 8.5.4.2]{HA}.  The method of constructing $\cA$, by splitting homotopy idempotents in a group ring, occurs again in our discussion of projective $\bA[G]$-modules, so let us indicate it here.  (The uniqueness of $\cA$ requires a more complicated argument).

Suppose $\cO = (\pi_0\bA)(\zeta_m)$, where $\zeta_m$ is a primitive $m$th root of unity and $m$ is prime to $p$.  To construct an extension $\cA$ of $\bA$ with the described properties, observe that if $C$ is any finite commutative group, the group structure on $C$ induces an $\mathrm{E}_{\infty}$-ring spectrum structure on $\bA[C] = \bigoplus_{c \in C} \bA$, whose induced structure on $\pi_0(\bA[C])$ is the usual ring structure on the group algebra $(\pi_0\bA)[C]$.  Since $m$ is prime to $p$, this group ring has an idempotent $e$ with $\pi_0 \bA[C] e = \cO$.  The idempotent provides a map of $\bA$-module spectra $\bA[C] \to \bA[C]$, and we let $\cA$ denote the homotopy limit of the sequence
\[
\bA[C] \stackrel{e}{\to} \bA[C] \stackrel{e}{\to} \cdots
\]
To see that $\cA$ is endowed with an $\mathrm{E}_{\infty}$-ring structure, note that we have canonically $\cA \cong L_{\cA} \bA[C]$ where $L_{\cA}$ denotes Bousfield localization.
\end{proof}

\subsection{$p$-completion}
Suppose $\bA$ is $p$-adic, and that $\varpi \in \pi_0 \bA$ is a uniformizer.  Let us say that $M \in \Mod(\bA)$ is 
\begin{itemize}
\item \emph{$L_{\hat{p}}$-acyclic} if $\varpi$ induces an isomorphism on $\pi_* M$
\item \emph{$p$-complete} (also called \emph{$L_{\hat{p}}$-local}) if $[N,M] = 0$ whenever $N$ is $L_{\hat{p}}$-acyclic.
\end{itemize}
For instance, $M$ is $p$-complete whenever $\pi_* M$ is finitely generated over $\pi_0 \bA$.  The $L_{\hat{p}}$-acyclic $\bA$-modules form a localizing subcategory and the theory of Bousfield localization produces an idempotent functor $L_{\hat{p}}:\Mod(\bA) \to \Mod(\bA)$ whose image is the full subcategory of $p$-complete modules.  ``Idempotent'' here means that $L_{\hat{p}}$ is equipped with a natural map $M \to L_{\hat{p}}M$ which is an isomorphism when $M$ is already $p$-complete.  The full subcategory of $p$-complete modules is denoted $L_{\hat{p}} \Mod(\bA)$.

\subsection{$\bA[G]$-modules}  Let $\bA$ be a good coefficient algebra.  By an $\bA[G]$-module we mean a functor from $BG$ to $\Mod(\bA)$ --- these form an $\infty$-category $\Fun(BG,\Mod(\bA))$.  We will also denote this category by $\LMod(\bA[G])$.  The Schwede-Shipley theorem \cite{SS}, \cite[\S7.1.2]{HA} shows that the notation is justified: there is an $\bA$-algebra spectrum $\bA[G]$ --- which we take to be $\bA[G] := \bA \wedge \Sigma^{\infty} G_+$ --- such that $\LMod(\bA[G])$ is the category of left modules over it.  Let us highlight some obvious constructions one can make with $\bA[G]$-modules, together with some explanation of how these constructions are carried out formally using the technology of \cite{HTT,HA}:

\subsubsection{Forgetful functor}  There is a forgetful functor $\LMod(\bA[G]) \to \Mod(\bA)$.  Formally, it is given by precomposing with the basepoint map $\mathit{pt} \to BG$.  As the basepoint map is essentially surjective, the forgetful functor is conservative.  When $\bA$ is $p$-adic we define the Bousfield localization $L_{\hat{p}}$ on $\LMod(\bA[G])$ by setting the $L_{\hat{p}}$-acyclic, resp. $p$-complete  modules to be those whose image in $\Mod(\bA)$ is $L_{\hat{p}}$-acyclic, resp. $p$-complete.

\subsubsection{The trivial module}  There is a trivial $\bA[G]$-module structure on $\bA$ itself, which we also denote by $\bA$.  Formally, it is the composition of the terminal map $BG \to \mathit{pt}$ with the map $\mathit{pt} \to \Mod(\bA)$ that picks out the object $\bA$.  Any object of $\Mod(\bA)$ can be given a trivial $G$-action in this way.

\subsubsection{$\delta$-functor}  There is a $2$-periodic $\delta$-functor on the homotopy category of $\LMod(\bA[G])$, taking values in the abelian category of $\pi_0\bA[G]$-modules, that we denote by $\pi_*$.

We denote the full subcategory of $\LMod(\bA[G])$ spanned by finite type modules --- modules $M$ whose homotopy groups $\pi_0 M,\pi_1 M$ are finitely generated over $\pi_0 \bA$  --- by $\LMod(\bA[G])^{\mathit{ft}}$.  In the present setting $\LMod(\bA[G])^{\mathit{ft}}$ is equivalent to the category of functors from $BG$ to $\Mod(\bA)^{\omega}$, what Mathew denotes by $\mathrm{Rep}(G,\bA)$ in the appendix.

\subsubsection{Tensor structure} The category $\LMod(\bA[G])$ inherits a symmetric monoidal structure from the symmetric monoidal structure on $\Mod(\bA)$.  Formally \cite[Def. 2.0.0.7]{HA}, the symmetric monoidal structure on $\Mod(\bA)$ is an identification of $\Mod(\bA)$ with the fiber of a suitable coCartesian fibration $\Mod(\bA)^{\otimes} \to \mathcal{F}\mathrm{in}_*$ above $\{0\}_*$, and we endow $\LMod(\bA[G])$ with a similar structure by taking $\LMod(\bA[G])^{\otimes}$ to be $\Fun(BG,\Mod(\bA)^{\otimes})$.

\subsubsection{Transformation and permutation modules}
\label{subsubsec:tandp}
If $X$ is a space with a left $G$-action, then $\bA[X]$ (the smash product of $\bA$ with the suspension spectrum of $X$) is naturally an object of $\LMod(\bA[G])$.  
There is a natural identification of $[\bA[X],\Sigma^i \bA]$ with the $G$-equivariant (Borel-style) $\bA$-cohomology of $X$, i.e. $\mathbf{A}^i(X_{hG})$.

We say that $\bA[X]$ is a \emph{transformation module.}  If $X$ is a finite $G$-set we say that $\bA[X]$ is a \emph{permutation module.}  When $Y$ is a $G$-stable subspace of $X$, we write $\bA[X,Y]$ for the cone on the induced map $\bA[Y] \to \bA[X]$.  Thus $[\bA[X,Y],\Sigma^i\bA]$ is identified with the $\bA$-cohomology of the pair $(X_{hG},Y_{hG})$.

\subsubsection{Induction and restriction}  If $H$ is a subgroup of $G$, precomposition with $BH \to BG$ gives a symmetric monoidal functor $\Fun(BG,\Mod(\bA)) \to \Fun(BH,\Mod(\bA))$.  We denote this functor by $\Res^G_H$.  It has a two-sided adjoint, that we denote by $\Ind_H^G$.  We have in particular $\Ind_H^G \bA[X] = \bA[G \times_H X]$.  

There is also a Mackey or base-change formula.  If $H$ and $K$ are two subgroups of $G$, then there is a natural isomorphism
\begin{equation}
\label{eq:Mackey}
\Res_K^G \Ind_H^G \stackrel{\sim}{\to} \bigoplus_x \Ind_{K \cap xH x^{-1}}^K \circ \mathrm{ad}_x \circ \Res_{x^{-1} K x \cap H}^H
\end{equation}
where the sum runs over double coset representatives for $K \backslash G /H$.  Formally, one can deduce \eqref{eq:Mackey} by appealing to the bimodule formalism we discuss next.

\subsubsection{Functors from bimodules}
\label{subsubsec:bimod}
  If $G_1$ and $G_2$ are finite groups, let us say that a $(G_1,G_2)$-bimodule is a left $\bA[G_1 \times G_2^{\op}]$-module.  There is a full embedding of this bimodule category into the functor category $\Fun(\LMod(\bA[G_2]),\LMod(\bA[G_1]))$, given by tensoring on the right \cite[Prop. 4.1]{BZFN}, \cite[\S4.8.4]{HA}.  It identifies the bimodule category with the full subcategory of continuous (i.e. colimit-preserving) functors.  
  
If $B$ is a $(G_1,G_2)$-bimodule and $C$ is a $(G_2,G_3)$-bimodule, their tensor product over $\KU_q[G_2]$ can be computed as the homotopy $G_2$-coinvariants of their tensor product over $\KU_q$, i.e.
\[
B \otimes_{\KU_q[G_2]} C \cong (B \otimes_{\KU_q} C)_{hG_2}
\]
where on the right $B \otimes_{\KU_q} C$ is regarded as a functor from $BG_2$ to the $\infty$-category of $(G_1,G_3)$-bimodules, using the diagonal inclusion $G_2 \to G_1 \times G_2^{\op} \times G_2 \times G_3^{\op}$.

If a $(G_1,G_2)$-bimodule is compact as a right $\bA[G_2]$-module, the corresponding functor carries $\LMod(\bA[G_2])^{\mathit{ft}}$ to $\LMod(\bA[G_1])^{\mathit{ft}}$.  (This is not true for a general bimodule --- even one of finite type.  When $\bA$ is the $p$-complete complex $K$-theory spectrum, it is closer to true.  We discuss this in  \S\ref{sec:Tate-van}.)

For example, given a finite group and subgroup $G \supset H$, the restriction and induction functors are represented by the permutation bimodule $\bA[G]$, where $G$ is endowed with a left action of $G$ and right action of $H$ or vice versa.  One can use this to deduce \eqref{eq:Mackey}: as a $(K,H)$-bimodule, the composition on the left is the tensor product $\bA[G] \otimes_{\bA[G]} \bA[G] \cong \bA[G]$, which is the permutation module given by the $K \times H^{\op}$-set $G$.

\subsection{Projective $\bA[G]$-modules}
\label{subsec:proj}

The homotopy category of left $\bA[G]$-modules has a big chunk familiar from the algebra of ordinary $\pi_0 \bA[G]$-modules: each projective $\pi_0\bA[G]$-module can be lifted in a unique way to an $\bA[G]$-module spectrum, and the homomorphisms between them are the same at least at the level of homotopy classes.

\begin{theorem*}
For $M \in \LMod(\bA[G])$, the following are equivalent:
\begin{enumerate}
\item $\pi_* M$ is a free $\pi_*\bA[G]$-module (resp. a projective $\pi_* \bA[G]$-module).
\item $M$ is isomorphic to a module of the form $\bA[G]^{\oplus i} \oplus \Sigma \bA[G]^{\oplus j}$ (resp. isomorphic to a summand of such a module).
\end{enumerate}
Write $\Free(\bA[G]) \subset \Proj(\bA[G]) \subset \LMod(\bA[G])$ for the full subcategories spanned by the free and projective objects.  Then $\pi_*$ induces an equivalence from the homotopy category of $\Free(\bA[G])$ onto the category of $\bZ/2$-graded free $\pi_0\bA[G]$-modules, and from the homotopy category of $\Proj(\bA[G])$ onto the category of $\bZ/2$-graded projective $\pi_0\bA[G]$-modules.
\end{theorem*}

\begin{proof}
For free modules, the theorem follows (as in the proof of Proposition \ref{thm:coefficients})
from the observation that $\pi_i M \cong [\Sigma^i \bA[G],M]$.  Since $\pi_*:[M,N] \to \Hom(\pi_*M,\pi_*N)$ is an isomorphism whenever $M$ is free, it follows that idempotents in $\Hom(\pi_*M,\pi_*M)$ and in $[M,M]$ are in one-to-one correspondence.  But idempotents split in $\LMod(\bA[G])$.  
\end{proof}

\subsection{Projective covers}
\label{subsec:cov}

If $P \in \LMod(\bA[G])$ is projective, it is easy to compute $[P,M]$ for any $\bA[G]$-module $M$:

\begin{prop*}
Let $M,P \in \LMod(\bA[G])$, and suppose $P$ is projective.  Then $[P,M]$ is isomorphic to $\Hom(\pi_*(P),\pi_*(M))$, where $\Hom$ is taken in the category of $\bZ/2$-graded $\pi_0\bA[G]$-modules.
\end{prop*}

\begin{proof}
For some $\bA[G]$-module $Q$, $P\oplus Q$ is free, and we can always extend a map $\pi_* P \to \pi_* M$ to a map $\pi_* P \oplus \pi_* Q \to \pi_* M$.  It therefore suffices to prove the proposition in the case that $P$ is free.  But for a free module we have $[P,M] \cong \Hom(\pi_*(P),\pi_*(M))$.
\end{proof}

Recall that a \emph{projective cover} of a $\pi_0 \bA[G]$-module $\overline{M}$ is a surjective map $\overline{f}:\overline{P} \to \overline{M}$ such that $\overline{f}$ and such that the restriction of $\overline{\pi}$ to any proper submodule fails to be surjective.  If they exist, projective covers are unique up to isomorphism.  (For example, if $\pi_0 \bA$ is a complete discrete valuation ring then, being free of finite rank over $\pi_0 \bA$, the ring $\pi_0\bA[G]$ is semiperfect in the sense of \cite[\S 2.1]{Bass}.  One characterization of semiperfect rings from loc. cit. is that every finitely generated left module has a projective cover.)

The discussion in this and the previous section shows that, when $\overline{M} = \pi_* M$ for $M \in \LMod(\bA[G])$, any $\bZ/2$-graded projective cover $\overline{f}:\overline{P} \to \overline{M}$ is the image under $\pi_*$ of a unique homotopy class of maps $f:P \to M$ in $\LMod(\bA[G])$.

\subsection{Blocks}
\label{subsec:blocks}
Let $\bA$ be a good algebra, and suppose that $\pi_0 \bA$ is a complete discrete valuation ring of mixed characteristic $p$.  For $\pi_0 \bA[G]$, as for any semiperfect ring, the following are in one-to-one correspondence:
\begin{enumerate}
\item Decompositions of $1 \in \pi_0\bA[G]$ as $1 = b_1 + \cdots + b_k$, where the $b_i$ are central idempotents.
\item Decompositions of the additive category of $\pi_0 \bA[G]$-modules as a direct sum $\cC_1 \oplus \cdots \oplus \cC_k$
\item Decompositions of the additive category of projective $\pi_0 \bA[G]$-modules as a direct sum $\cP_1 \oplus \cdots \oplus \cP_k$.
\end{enumerate}
The correspondence is obtained by taking for $\cC_i$ (resp. $\cP_i$) the full subcategory of modules (resp. projective modules) on which $b_i$ acts as the identity.  If the $b_i$ are primitive, they are called ``block idempotents'' and the categories $\cC_i \supset \cP_i$ are called ``blocks.''  The block categories $\cC_i$ and $\cP_i$ are indecomposable: they cannot be written as a direct sum of two nonzero additive categories.

We can transport these notions to $\bA[G]$, by the following device.  If $b$ is any central idempotent in $\pi_* \bA[G]$, let us say that $M$ is \emph{$b$-acyclic} if the induced map $\pi_* M \to \pi_* M$ is zero.

\begin{prop*}
The $b$-acyclic objects of $\LMod(\bA[G])$ are the acyclic objects for a Bousfield localization $L_b$ of $\LMod(\bA[G])$.  The essential image of $L_b$ is the full subcategory of $(1-b)$-acyclic objects, and the pair $L_b,L_{1-b}$ induce an equivalence
\[
\LMod(\bA[G]) \cong L_b \LMod(\bA[G]) \oplus L_{1-b}\LMod(\bA[G])
\]
\end{prop*}

\begin{proof}
The class of $b$-acyclic objects is closed under infinite direct sums, since $\pi_*(\bigoplus M_i) \cong \bigoplus \pi_*(M_i)$.  If $M''$ is the cone on a map $M' \to M$, and both $M'$ and $M$ are $b$-acyclic, then by the long exact sequence in homotopy $b$ must act nilpotently on $\pi_* M''$.  If $b$ acts nilpotentely and idempotently, then $b$ is zero.  This shows that the $b$-acyclic objects form a localizing subcategory of $\LMod(\bA[G])$.  
\end{proof}

If $b_1,\ldots,b_n$ are the block idempotents of $\pi_0 \bA[G]$, the proposition gives us a decomposition of stable $\infty$-categories
\[
\LMod(\bA[G]) \cong L_{b_1} \LMod(\bA[G]) \oplus \cdots \oplus L_{b_n} \LMod(\bA[G])
\]
\begin{prop*}
Let $b_i$ be a block idempotent of $\pi_0 \bA[G]$.  Then $L_{b_i} \LMod(\bA[G])$ is indecomposable, in the sense that it cannot be written as a direct sum of two nonzero stable $\infty$-categories.
\end{prop*}
\begin{proof}
We will show that if $L_{b_i} \LMod(\bA[G]) \cong \cC_1 \oplus \cC_2$, then either $\cC_1$ or $\cC_2$ is the zero category.  Under the equivalence of \S\ref{subsec:proj} between the homotopy category of $\Proj(\bA[G])$, and the category of $\bZ/2$-graded projective modules over $\pi_*\bA[G]$, the $L_{b_i}$-local projective $\bA[G]$-modules map to $\bZ/2$-graded projective modules over $\pi_*\bA[G]b_i$.  

Now suppose $L_{b_i} \LMod(\bA[G]) \cong \cC_1 \oplus \cC_2$.  Then, since the category of projective $\pi_*\bA[G]b_i$-modules is indecomposable, projective objects in $L_{b_i} \LMod(\bA[G])$ are a subset of the objects of either $\cC_1$ or $\cC_2$.  Suppose it is $\cC_1$, and let us show that $\cC_2$ is zero.  Indeed any $M \in \cC_2$ must have $[P,M] = 0$ for every projective $P$, by \S\ref{subsec:cov}.
\end{proof}

The Theorem of \S\ref{intro:blocks} is an immediate consequence of the previous two propositions, i.e.:

\begin{theorem}
Suppose $\bA$ is a good algebra and $\pi_0\bA$ is a complete discrete valuation ring of mixed characteristic $p$.  Then $M$ belongs to a block of $\LMod(\bA[G])$ if and only if $\pi_0(M)$ and $\pi_1(M)$ belong to a single block of $\pi_0\bA[G]$-modules.  The induced map on blocks is a bijection.
\end{theorem}

\subsection{Adams filtration}
\label{subsec:adamsfilt}
Given $M,N \in \LMod(\bA[G])$, let $[M,N]^{\geq k}$ denote the group of homotopy classes of maps that can be factored as 
\[
M \to U_1 \to U_2 \to \cdots \to U_k \to N
\]
with the property that $\pi_*(M) \to \pi_*(U_1)$ and $\pi_*(U_i) \to \pi_*(U_{i+1})$ for $i \leq k - 1$ vanish.  We may describe this filtration alternatively using a projective resolution of $M$.

Let us call a sequence of homotopy classes of maps of $\bA[G]$-modules
\[
M \leftarrow P_0 \leftarrow P_1 \leftarrow P_2 \leftarrow \cdots
\]
a projective resolution of $M$ if each $P_i$ is projective and the sequence induces a projective resolution on $\pi_*$.  By the discussion of \S\ref{subsec:proj}--\S\ref{subsec:cov}, isomorphism classes of projective resolutions of $M$ are in one-to-one correspondence with isomorphism classes of $\bZ/2$-graded projective resolutions of the $\pi_0 \bA[G]$-module $\pi_* M$.  We define  modules $Y_i$ and maps $P_i \to Y_i \to P_{i-1}$ inductively, as indicated in the following diagram:
\[ 
\xymatrix{
& & Y_1 \ar[dl]  & & & & Y_3 \ar[dl] \\
& P_0 \ar[dl] & &  P_1 \ar[ul] \ar[ll] & & P_2 \ar[dl] \ar[ll] & & \ar[ll] \ar[ul]  \cdots \\
M & & & & Y_2 \ar[ul]
}
\] 
where each diagonal line is a fiber sequence.  As $\bA$ is a good coefficient algebra, the $Y_i$ are all $\bA[G]$-modules whose underlying $\bA$-module is free.  Each map $P_i \to Y_i$ induces a surjection on homotopy groups.  The connecting maps $Y_i \to \Sigma Y_{i+1}$ assemble to a sequence
\[
M \to \Sigma Y_1 \to \Sigma^2 Y_2 \to \cdots
\]
that is called the ``Adams tower'' of the projective resolution.

\begin{prop*}
Let $M \leftarrow P_{\bullet}$ be a projective resolution of a $\bA[G]$-module $M$, and let 
\[
M \to \Sigma Y_1 \to \Sigma^2 Y_2 \to \cdots
\]
be the induced Adams tower.  Let $N$ be any other $\bA[G]$-module.  Then $[M,N]^{\geq k}$ is the image of $[\Sigma^k Y_k,N]$ in $[M,N]$.
\end{prop*}

\begin{proof}
By definition, $f \in [M,N]$ is in the image of $[\Sigma^k Y_k,N]$ if and only if there is a map $f_k:\Sigma^k Y_k \to N$ such that $f$ is homotopic to the composite
\[
M \to \Sigma^1 Y_1 \to \cdots \to \Sigma^k Y_k \to N
\]
This establishes that $[\Sigma^k Y_k,N]$ maps to $[M,N]^{\geq k}$.  To show the reverse inclusion, suppose we have a factorization $M \to U_1 \to \cdots \to U_k \to N$ and let us show that we can fill in the dotted arrows in the diagram
\[
\xymatrix{
M \ar[r] \ar[dr] & \Sigma^{1} Y_1 \ar[r] \ar@{-->}[d] & \cdots \ar[r] & \Sigma^{k} Y_k  \ar@{-->}[d] \\
& U_1 \ar[r] & \cdots \ar[r] & U_k  \ar[r] & N
}
\]
Let $Q_0$ be the homotopy fiber of the map $M \to U_1$.  Since $M \to U_1$ is zero on homotopy groups, $Q_0 \to M$ is surjective on homotopy groups.  It follows from the projectivity of $P_0$ that we may find a map $P_0 \to M$ such that the diagram
\[
\xymatrix{
P_0 \ar[d] \ar[r] & M \ar@{=}[d] \ar[r] & \Sigma^{1} Y_1  \\
Q_0 \ar[r] & M  \ar[r] & U_1  
}
\]
commutes and has exact rows.  The existence of $\Sigma^1 Y_1 \to U_1$ follows.  We repeat this construction $k$ times until we reach a map $\Sigma^k Y_k \to U_k$.
\end{proof}

\begin{example*}
The trivial module $\bA$ can be identified with the transformation module $\bA[EG]$, where $G$ is the tautological $G$-bundle over the classifying space $BG$.  Let $BG_{\leq n}$ be the $n$-skeleton of the standard CW structure on $BG$, and $EG_{\leq n}$ its preimage.  Then $\Sigma^{-n} \bA[EG_{\leq n},EG_{\leq n-1}]$ is a free $\bA[G]$-module with generators in bijection with the $n$-cells of $BG$.  The Adams tower is
\begin{equation}
 \bA[EG] \to \bA[EG,EG_{\leq 0}] \to \bA[EG,EG_{\leq 1}] \to \cdots
\end{equation}
Thus, the Adams filtration of $[\bA,\bA] \cong \bA^i(BG)$ is the usual Atiyah-Hirzebruch filtration --- the $n$th filtered piece is the kernel of the restriction map $\bA^*(BG) \to \bA^*(BG_{\leq n-1})$.  More generally, this shows that the Adams filtration on $[\bA,M] = \pi_0(M^{hG})$ is Hausdorff.
\end{example*}

\subsection{Adams spectral sequence}
Let $M \leftarrow P_{\bullet}$ be a projective resolution of an $\bA[G]$-module $M$, as in \S\ref{subsec:adamsfilt}.    Let $M \to \Sigma^1 Y_1 \to \Sigma^2 Y_2 \to \cdots$ be the associated Adams tower.  Define bigraded groups
\[
A^{s,t} = [\Sigma^{-t} Y_s,N] \qquad E_1^{s,t} = [\Sigma^{-t} P_s,N]
\]
where $Y_0 = M$ and $Y_s = P_s = 0$ for $s < 0$.  We evidently have an exact couple
\[
\xymatrix{
A \ar[rr]^i & & A \ar[dl]^j \\
& E_1 \ar[ul]^k
}
\qquad
\begin{array}{c}
\\ \\ \\ \\
\text{i.e.}
\end{array}
\qquad
\xymatrix{
A^{s+1,t} & A^{s+1,t-1} \ar[rr] & & A^{s,t} \ar[dl] \\
& & E_1^{s,t} \ar[ull]
}
\]
This leads to a spectral sequence $(E_r,d_r)$ where $d_r$ has degree $(r,1-r)$.  On the first page the differential  $E_1^{s,t} \to  E_1^{s+1,t}$ is induced by the map $P_{s+1} \to P_s$.  Since $\pi_* P_{\bullet} \to \pi_* M$ is a projective resolution, we have
\begin{equation}
\label{eq:E2st}
E_2^{s,t} = \Ext^s(\pi_* M,\pi_{* - t} N) 
\end{equation}

The $\Ext$ is taken in the abelian category of $\bZ/2$-graded $\pi_0\bA[G]$-modules.  
When $M = \bA$, the spectral sequence converges --- in fact the $E_{\infty}$ page coincides with the associated graded of the Adams filtration on $[\bA,\Sigma^{s+t} N]$, and \eqref{eq:E2st} is known as the homotopy fixed point spectral sequence.
It would be desirable to know more general conditions under which this is true, it appears this cannot be taken for granted \cite[\S 8.2.1.24]{HA}. 

\section{Representations on $\KU$-modules}
In this section, we specialize to the case where the good coefficient algebra is the $p$-completed $K$-theory spectrum, or a cyclotomic extension of it.  We denote the complex $K$-theory spectrum by $\KU$, its $p$-completion by $\KU_p$, and the cyclotomic extensions of $\KU_p$ by $\KU_q$, where $q$ is a power of $p$.

\subsection{The theorems of Atiyah and Kuhn}
\label{sec:AtiyahKuhn}

If $X$ is a $G$-space, let $K_G(X)$ denote the Grothendieck ring of $G$-equivariant complex vector bundles on $X$.  Let $I_G^{\mathrm{Aug}}(X) \subset K_G(X)$ denote the augmentation ideal --- the virtual vector bundles of dimension zero.  The Atiyah-Segal completion theorem \cite{AtiyahSegal} states that the natural map $K_G(X) \to \KU^0(X_{hG})$ induces an isomorphism
\begin{equation}
\label{eq:AtiyahSegal}
\left(\text{$I_G^{\mathrm{Aug}}$-adic completion of }K_G(X)\right) \stackrel{\sim}{\to} \KU^0(X_{hG})
\end{equation}
For simplicity let us suppose that $X$ is a finite set, so that $X_{hG}$ is a disjoint union of classifying spaces of subgroups of $G$.  Then after completing at $p$, the identification 
\eqref{eq:AtiyahSegal} becomes even more concrete: we have 
\begin{equation}
\label{eq:Kuhn}
\KU_q^0(X_{hG}) \cong K_G(X) \otimes \bZ_q/I_G^{\mathrm{Kuhn}}(X) \cong K_G(X)/I_G^{\mathrm{Kuhn}}(X) \otimes \bZ_q
\end{equation}
and $\KU_q^1(X_{hG})$ vanishes.  Here $I_G^{\mathrm{Kuhn}}$ denotes the ideal of virtual $G$-equvariant vector bundles that vanish under the restriction map $K_G(X) \to K_P(X)$ for any $p$-subgroup of $G$.

\begin{lemma}
\label{lem:thomas-hart-benton}
For any $M \in \LMod(\KU_q[G])^{\mathit{ft}}$, the group $[\bA,M]$ is finitely generated over $\bZ_q$.
\end{lemma}

\begin{proof}
We have $[\bA,M] \cong \pi_0 M^{hG}$.  If $P \subset G$ is a Sylow $p$-subgroup, then $M^{hG}$ is a summand of $M^{hP}$, and if $P' \subset P$ is a normal subgroup then $M^{hP} = (M^{hP'})^{h(P/P')}$.  These remarks reduce us to the case when $G$ is cyclic of order $p$.

Let us prove the lemma by showing that $[\bA,M]$ is finitely generated as a module over the finite $\bZ_q$-algebra $[\bA,\bA] \cong R(G) \otimes \bZ_q$.  Since $\pi_* M$ is finitely generated over $\bZ_q$, $H^*(G,\pi_* M)$ is finitely generated over $H^*(G,\bZ_q)$ \cite[Prop. 4.1]{Evens}.  This is the $E_2$ page of the homotopy fixed points spectral sequence for $M$, which is a module over the Atiyah-Hirzebruch spectral sequence for $G$.  For a cyclic group, the latter spectral sequence degenerates on the second page \cite[Prop 8.1]{Atiyah}, so the $E_{\infty}$ page of the homotopy fixed points spectral sequence for $M$, i.e. the associated graded of the Adams filtration on $[\bA,M]$, is also a module over $H^*(G,\bZ_q) = \mathrm{gr}(R(G) \otimes \bZ_q)$.  The finite generation now follows from \cite[III.2.9 Cor. 1]{Bourbaki}.
\end{proof}

\subsection{Transformation and permutation modules}
With $X$ a $G$-space, we construct a transformation module $\KU_q[X] \in \LMod(\KU_q[G])$ as in \S\ref{subsubsec:tandp}.  The Borel-equivariant $\KU_q$-theory of $X$ computed by \eqref{eq:AtiyahSegal} is naturally encoded as a hom set in the homotopy category of $\LMod(\KU_q[G])$
\begin{equation}
\label{eq:gencoh}
[ \KU_q[X],\Sigma^i \KU_q] \cong \KU_q^i(X_{hG})
\end{equation}
Direct sums and tensor products of transformation modules are also transformation modules.
\begin{equation}
\label{eq:ringhom}
\KU_q[X \times Y] \cong \KU_q[X] \otimes \KU_q[Y] \qquad \KU_q[X \amalg Y] \cong \KU_q[X] \oplus \KU_q[Y]
\end{equation}

Now suppose $X$ is a finite set, i.e. $\KU_q[X]$ is a permutation module.  Then we have a Kronecker pairing $\delta:\KU_q[X] \otimes \KU_q[X] \to \KU_q$.  As $X \times X$ is the disjoint union of the diagonal copy of $X$ and its complement, \eqref{eq:ringhom} gives $\KU_q[X]$ as a canonical summand of $\KU_q[X] \otimes \KU_q[X]$.  The map $\delta$ is the composite with the map to $\KU_q$ induced by $X \to \mathit{pt}$.

For any $M,N \in \LMod(\KU_q[G])$, the Kronecker $\delta$ induces a homotopy equivalence
\begin{equation}
\label{eq:zorro}
\Maps(M,N \otimes \KU_q[X]) \stackrel{\sim}{\to} \Maps(M \otimes \KU_q[X],N)
\end{equation}
given by first applying $\otimes \KU_q[X]$, and then composing with $N \otimes \KU_q[X] \otimes \KU_q[X] \stackrel{1 \otimes \delta}{\longrightarrow} N$.  In other words, it exhibits $\KU_q[X]$ as \emph{self-dual} in the sense of monoidal categories.
Applying \eqref{eq:zorro} with $M = \KU_q[Y]$ and $N = \KU_q$ gives an identification 
\begin{equation}
\label{eq:for19}
[\KU_q[X],\KU_q[Y]] \cong \KU_q^0((Y \times X)_{hG}).
\end{equation}
In particular (after \eqref{eq:Kuhn}) there is a surjective map $K_G(Y \times X) \otimes \bZ_q \to [\KU_q[X],\KU_q[Y]]$.  The kernel consists of those virtual vector bundles that are Kuhn trivial in the sense of \S\ref{intro:Kuhn}.  The composition law in the category of $\KU_q[G]$-modules is compatible with the convolution of virtual vector bundles, (which we review in \S\ref{subsec:42}, see \eqref{eq:conv}).

\begin{prop*}
If $Z$ is a third finite $G$-set, the composition and convolution maps
commute with the identifications of \eqref{eq:zorro}.  That is, the diagram
\[
\xymatrix{
(K_G(X \times Y) \otimes \bZ_q) \times (K_G(Y \times Z) \otimes \bZ_q) \ar[r] \ar[d] & K_G(X \times Z) \otimes \bZ_q \ar[d] \\
[\KU_q[Y],\KU_q[X]] \times [\KU_q[Z],\KU_q[Y]] \ar[r] & [\KU_q[Z],\KU_q[X]]
}
\]
commutes.
\end{prop*}

The verification is standard, e.g. \cite[Rem. 4.11]{BZFN} for a much more sophisticated version.

\subsection{One-dimensional modules}
\label{subsec:one-dimensional}
If $\bA$ is a good coefficient algebra, let us say an object $M \in \LMod(\bA[G])$ is one-dimensional, or invertible, if $\pi_0(M) = \pi_0(\bA)$ and $\pi_1(M) = 0$, or equivalently if the underlying $\bA$-module is isomorphic to $\bA$ itself.  The units in $\bA$ are a loop space (in fact an infinite loop space) called $\GL_1(\bA)$, and a $1$-dimensional module is given by a map $\rho:BG \to B\GL_1(\bA)$.

\begin{example*}
When $\bA = \KU$, one may construct $1$-dimensional modules from actions of $G$ on the additive category of finite-dimensional complex vector spaces --- let us denote this category by $\mathrm{Vect}$.  The action of $G$ is automatically continuous for the natural topology on the hom spaces in $\mathrm{Vect}$.  The construction of $\KU$ as the group completion (followed by inverting the Bott element)
\[
\Vect^{\simeq} \mapsto \bZ \times B\mathrm{U} \mapsto \KU
\]
is functorial, so the action on $\mathrm{Vect}$ induces an action on $\KU$, i.e. a one-dimensional module.

There are no covariant self-equivalences of $\mathrm{Vect}$ besides the identity functor, but this functor has a $\mathbf{C}^*$ of automorphisms --- the automorphism group (or ``2-group'') of $\mathrm{Vect}$ is a delooping $B\bC^*$ of $\bC^*$, and its classifying space is a second delooping $B\! B \bC^*$.  An action of $G$ on $\mathrm{Vect}$ is therefore given by a basepoint-preserving map $BG \to B\! B\bC^*$ --- in particular if $G$ is finite such actions are classified by $H^2(BG;\mathbf{C}^*)$. 
\end{example*}

For given $\mathbf{A}$, there is some hope of computing the complete set of $\rho$:  since $\GL_1(\bA)$ is an infinite loop space, it represents an extraordinary cohomology theory denoted $H^*(-;\mathrm{gl}_1(\bA))$.  The set of isomorphism classes of invertible modules is $H^1(BG;\mathrm{gl}_1(\bA))$.  The modules of the Example arise from a splitting of the exact triangle of spectra
\[
\tau_{\geq 4} \gl_1(\KU) \to \gl_1(\KU) \to \tau_{\leq 2} \gl_1(\KU); \qquad \gl_1(\KU) \leftarrow \tau_{\leq 2} \gl_1(\KU)
\]
When $\KU$ is replaced by $\KU_p$ the splitting on the right no longer exists, but the spectrum $\gl_1(\bA)$ can be analyzed using the logarithm of \cite{Rezk}.  In particular one may prove the following:

\begin{prop*}
Suppose $G$ is a cyclic group of $p$-power order, or else that $G$ is elementary abelian and $p \geq 5$.
Then the natural map
\begin{equation}
\label{eq:truncationmap}
H^1(BG;\gl_1(\KU_p)) \to H^1(\tau_{\leq 2} \gl_1(\KU_p))
\end{equation}
is an injection.
\end{prop*}

\begin{proof}
Here $\tau_{\leq 2}$ denotes the cone on the map from the 2-connective cover of $\gl_1(\KU_p)$ --- as the 2-connective cover is also the 3-connective cover, we denote it by $\tau_{\geq 4} \gl_1(\KU_p)$.

The exact triangle of spectra induces a long exact sequence on extraordinary cohomologies, and to show that \eqref{eq:truncationmap} is an isomorphism it suffices to show that
$H^1(BG;\tau_{\geq 4} \gl_1(\KU_p))$ vanishes.

As $\KU_p$ is a $K(1)$-local spectrum, one has a Rezk logarithm map $\gl_1(\KU_p) \to \KU_p$.  From the formula of \cite[Theorem 1.9]{Rezk}, one sees that this map is an isomorphism on homotopy groups in degrees $4$ and greater --- in particular it induces an isomorphism $\tau_{\geq 4} \gl_1(\KU_p) \to \tau_{\geq 4} \KU_p$.  Note the Bott map further identifies this with $\Omega^{-4} \tau_{\geq 0} \KU_p$.  Thus we are reduced to showing that $H^1(BG;\tau_{\geq 4} \KU_p) = H^5(BG;\tau_{\geq 0} \KU_p)$ vanishes.  According to \cite[p. 2]{Ossa}, the kernel of the map from connective $K$-theory to periodic $K$-theory is concentrated in degrees $\geq 2p-3$.
\end{proof}

The same technique gives a long exact sequence relating the low-degree connective $K$-theory of $BG$ to the one-dimensional modules of $\KU_p[G]$ that do not arise from $G$-actions on $\Vect$.

\subsection{Functors from finite type bimodules}
\label{sec:Tate-van}

Let $B \in \LMod(\KU_q[G_1 \times G_2^{\op}])^{\mathit{ft}}$ be a $(G_1,G_2)$-bimodule of finite type.  Then the composite functor
\[
\LMod(\KU_q[G_2]) \xrightarrow{\otimes_{\KU_q[G_2]} B} \LMod(\KU_q[G_1]) \xrightarrow{L_{\hat{p}}} \LMod(\KU_q[G_1])
\]
carries finite type $\KU_q[G_2]$-modules to finite type $\KU_q[G_1]$-modules.

To see this requires the main result of \cite{GreenleesSadofsky}.  Note that we can replace $L_{\hat{p}}$ with $L_{K(1)}$ --- as $\KU_q$ is itself a $K(1)$-local spectrum, any $K(1)$-local $\KU_q$-module spectrum is $p$-complete and vice versa.
If $M$ is a $\KU_q[G_2]$-module, we may identify $B \otimes_{\KU_q[G_2]} M$ with  $(B \otimes_{\KU_q} M)_{hG_2}$.  If $B$ and $M$ are of finite type, then so are $B \otimes_{\KU_q} M$ and $(B \otimes_{\KU_q} M)^{hG_2}$ (the latter after Lemma \ref{lem:thomas-hart-benton}).  Now the fact that $L_{\hat{p}} (B \otimes_{\KU_q[G_2]}M)$ has finite type follows from the formulation of the Greenlees-Sadofsky theorem given in \cite[Th. 0.0.1]{HopkinsLurie}: there is a natural identification
\begin{equation}
\label{eq:GrSa}
L_{K(1)}N_{hG} \stackrel{\sim}{\to} N^{hG}
\end{equation}
for any $K(1)$-local spectrum $N$.

\subsection{Commutative $p$-groups and Koszul transform}
\label{subsec:Koszul}
Let $G$ be a commutative $p$-group, and let $G^{\sharp} := \Hom(G,\bC^*)$ be its Pontrjagin dual.  If $G$ has exponent $p^e$, then so does $G^\sharp$.  Let $E$ denote the $\bC^*$-valued $2$-cocycle on $G \times G^{\sharp}$ given by the formula\footnote{There are other natural cocycles of the form $\phi(h)^a \psi(g)^b$, for $a,b \in \bZ/p^e$, but those of the form $\phi(h)^a \psi(g)^a$ are coboundaries.  They make a cyclic group of order $p^e$ inside $H^2(B(G \times G^{\sharp});\bC^*)$, the class of $E$ given is a generator.}
\begin{equation}
\label{eq:Ecocycle}
E(g,\phi,h,\psi) = \phi(h)
\end{equation}
This determines a class in $H^2(B(G \times G^\sharp);\bC^*)$ of exact order $p^e$, as one can see by considering its restriction to a subgroup generated by $(g,1)$ and $(1,\phi)$ where $\phi(g)$ is a primitive $p^e$th root of unity.

Let $M_E$ denote the one-dimensional $\KU_q[G \times G^{\sharp}]$-module associated to $E$, in the manner of the Example of \S\ref{subsec:one-dimensional}.  It is a $(G,G^\sharp)$-bimodule, (as $G$ is commutative $G^\sharp = (G^\sharp)^{\op}$) and determines a functor
$
\LMod(\KU_q[G^\sharp]) \to \LMod(\KU_q[G])
$ by tensoring with $M_E$.  After \S\ref{sec:Tate-van}, the formula $N \mapsto L_{\hat{p}}(M_E \otimes_{\KU_q[G^{\sharp}]} N)$ determines a functor 
\begin{equation}
\label{eq:i-prefer-Fourier}
\LMod(\KU_q[G^\sharp])^{\mathit{ft}} \to \LMod(\KU_q[G])^{\mathit{ft}}
\end{equation}

Let us prove the Theorem of 
\ref{intro:newsymmetries}:

\begin{theorem*}
The functor \eqref{eq:i-prefer-Fourier} is an equivalence.
\end{theorem*}

\begin{proof}
Let us show that the $p$-completed tensor product
\begin{equation}
\label{eq:thistensorproduct}
L_{\hat{p}}(M_{E_{G^\sharp}} \otimes_{\KU_q[G]} M_{E_G})
\end{equation}
is isomorphic to the diagonal $(G^\sharp,G^{\sharp})$-bimodule.  The tensor product $M_{E_G^{\sharp}} \otimes_{\KU_q[G]} M_{E_G}$ is the homotopy quotient by the diagonal $G$-action on $M_{E_{G^{\sharp}}} \otimes_{\KU_q} M_{E_G}$, so after \eqref{eq:GrSa}, it suffices to show that $(M_{E_{G^{\sharp}}} \otimes_{\KU_q} M_{E_G})^{hG}$ is isomorphic to the diagonal bimodule.  

As a $(G^\sharp \times G \times G^{\sharp})$-module, $M_{E_{G^\sharp}} \otimes_{\KU_q} M_{E_G}$ is the rank one module $M_{E'}$ arising from the $\bC^*$-valued cocycle given by
\begin{equation}
\label{eq:EEcocycle}
E':(\phi,g,\psi), (\chi,h,\omega) \mapsto \chi(g) \psi(h)
\end{equation}
When restricted to $G^\sharp \times G$ along the map $(\phi,g) \mapsto (\phi,g,\phi)$, \eqref{eq:EEcocycle} becomes $(\phi,g),(\omega,h) \mapsto \omega(g)\phi(h)$, a coboundary.  Thus there is an isomorphism $j:\KU_q \stackrel{\sim}{\to} \Res^{G^{\sharp} \times G \times G^{\sharp}}_{G^{\sharp} \times G} M_{E'}$, which is adjoint to a morphism
\begin{equation}
\label{eq:now-this}
\Ind_{G^\sharp \times G}^{G^{\sharp} \times G \times G^{\sharp}} \KU_q \to M_{E'}
\end{equation}
Note the domain of this map is isomorphic to the diagonal bimodule for $G^\sharp \times G^\sharp$, pulled back along the projection $G^\sharp \times G \times G^\sharp \to G^\sharp \times G^\sharp$.  In particular $G$ acts trivially on the domain, so by the universal property of homotopy fixed points \eqref{eq:now-this} factors through $M_{E'}^{hG}$.
By construction, the restriction of this map $\iota:\KU_q[G^{\sharp}] \to M_{E'}^{hG}$ along the diagonal $G^\sharp \to G^\sharp \times G^\sharp$ is isomorphic to the image under $(-)^{hG}$ of the map $j$ --- as $j$ is an isomorphism and restriction is conservative functor, the map $\iota$ is an isomorphism.

An identical argument shows that $L_{\hat{p}} M_{E_G} \otimes_{\KU_q[G^\sharp]} M_{E_{G^{\sharp}}}$ is isomorphic to the diagonal bimodule, completing the proof.
\end{proof}

\subsection{$K$-theoretic analog of a theorem of Goresky-Kottwitz-MacPherson}
\label{subsec:GKM}
If $T$ is a compact torus, $C^*(BT)$ denotes the real-valued cochain algebra and $C_*(T)$ denotes the real-valued chain algebra with its Pontrjagin product, there is a Koszul equivalence between a full subcategory of suitably finite objects in the derived category of $C^*(BT)$-modules and a similar full subcategory of  the derived category of $C_*(T)$-modules.  We regard the equivalence of \S\ref{subsec:Koszul} as a $K$-theoretic analog of this equivalence.  To explain why, let us prove the following

\begin{prop*}
Let $G$ be a commutative $p$-group.  Let $X$ be a $G$-space, and suppose that $\KU_q[X]$ has finite type.  Let $\KU_q[X] \in \LMod(\KU_q[G])$ be the corresponding transformation module, and let $M \in \LMod(\KU_q[G^{\sharp}])$ be its image under the equivalence of \S\ref{subsec:Koszul}.  Then the underlying $\KU_q$-module of $M$ is naturally identified with $L_{\hat{p}} \KU_q[X_{hG}]$.
\end{prop*}

This is analogous to \cite[Theorem 1.5.1]{GKM}.

\begin{proof}
Let $E \in H^2(G^\sharp \times G;\bC^*)$ be the cocycle given by $E(\phi,g,\psi,h) = \psi(g)$, let $M_E$ denote the corresponding $\KU_p[G^\sharp \times G]$-module.  The underlying $\KU_q$-module spectrum of the image of $\KU_q[X]$ in $\LMod(\KU_q[G^\sharp])$ is naturally identified with
$(M_E \otimes_{\KU_q} \KU_q[X])^{hG}$.  But as a $G$-module, $M_E$ is isomorphic to the trivial module, so this is naturally identified with $\KU_q[X]^{hG} \cong L_{\hat{p}}\KU_q[X_{hG}]$. 
\end{proof}

\subsection{The character ring $\Cent(\KU_q[G])$}
\label{subsec:19}
In \S\ref{intro:19}, we defined the character ring $\Cent(\KU_q[G])$ to be the endomorphism ring of the identity functor on the homotopy category of $\KU_q[G]$-modules.  A standard argument shows that this ring is commutative.  Below we will show that it is reduced and determine the structure of its prime ideal spectrum, but let us first explain the identification with $\KU_q^0(G_{\conj,hG})$, where $G_{\conj,hG}$ denotes the Borel construction of the conjugation action of $G$ on itself. 

By  \S\ref{subsubsec:bimod}, $\Cent(\KU_q[G])$ is naturally identified with the homotopy classes of self-maps of the identity bimodule $\KU_q[G] \in \LMod(\KU_q[G \times G^{\op}])$.  This is a permutation $G \times G^{\op}$-module, so by \eqref{eq:for19}, the group of self-maps is given by $\KU_q^0((G \times G)_{h(G \times G^{\op})})$, where $G \times G^{\op}$ acts on $G \times G$  by $((g,h),(x,y)) \mapsto (gxh,gy h)$.  The diagonal embedding $G_{\conj} \to G \times G$ is equivariant for the homomorphism $G \to G \times G^{\op}$ carrying $g$ to $(g,g^{-1})$, and gives an equivalence of groupoids $G_{\conj}/G \to (G \times G)/(G \times G^{\op})$, so that $\KU_q^0((G \times G)_{h(G \times G^{\op})}) \cong \KU_q^0(G_{\conj,hG})$.

By \eqref{eq:Kuhn}, we can identify $\KU_q^0(G_{\conj,hG})$ with the quotient of $K_G(G) \otimes \bZ_q$ by Kuhn-trivial virtual bundles.

\begin{proof}[Proof of the Theorem of \S\ref{intro:19}]
It follows from \S\ref{sec:AtiyahKuhn} that $\KU_q^0(G_{\conj,hG})$ is free of finite rank over $\bZ_q$.  In fact $\KU_q^0(G_{\conj,hG}) = \bigoplus_g \KU_q(B\Cent_G(g))$, where $g$ runs through conjugacy class representatives, so that the rank $r$ is the number of conjugacy classes of commuting pairs $(g,u)$ where $u$ has $p$-power order.  To prove that $\KU_q^0(G_{\conj,hG})$ is reduced, we fix an embedding $\bZ_q \hookrightarrow \bC$ and construct $r$ distinct $\bZ_q$-algebra homomorphisms $\KU_q^0(G_{\conj,hG}) \to \bC$.  It then follows by the linear independence of characters that $\KU_q^0(G_{\conj,hG})$ is a subring of $\prod_{i = 1}^r \bC$, and in particular is reduced.

Note that $r$ is also equal to the number of isomorphism classes of pairs $(u,L)$ where $u$ has $p$-power order and $L$ is an irreducible representation of $\Cent_G(u)$.  For each such $(u,L)$, if $\cE$ is a conjugation-equivariant virtual vector bundle on $G$ we define the quantity
\begin{equation}
\label{eq:chiuL}
\chi_{(u,L)}(\cE) := \frac{1}{\dim(L)} \sum_{g \in \Cent_G(u)} \trace(u\vert \cE_g) \trace(g \vert L)
\end{equation}
This is a ring homomorphism --- this can be checked directly, or see \cite[\S 2.2]{Lusztig1}. 

Any commutative ring that is free of finite rank over $\bZ_q$ has maximal ideals in bijection with blocks.  Any such commutative ring is moreover of Krull dimension one, and the association $\chi \mapsto \ker(\chi)$ is a surjection from the $\bC$-points to minimal primes, with two $\bC$-points determining the same prime if they are in the same orbit under $\mathrm{Gal}(\bC/\bQ_q)$.  To complete the proof, we observe that (assuming $q$ is chosen sufficiently large) $\gamma \in \mathrm{Gal}(\bC/\bQ_q)$ acts through $\mathrm{Gal}(\bQ_q(\mu_{p^{\infty}})/\bQ_q) \cong \bZ_p^*$, via:
\begin{eqnarray*}
\gamma(\chi_{(u,L)}) & = & \frac{1}{\dim(L)} \sum_{g \in \Cent_G(u)} \gamma(\trace(u\vert \cE_g)) \gamma (\trace(g\vert L)) \\
& = & \frac{1}{\dim(\gamma(L))} \sum_{g \in \Cent_G(u)} \trace(u^{\gamma} \vert \cE_g) \trace(g \vert \gamma(L)) \\
& = & \frac{1}{\dim(\gamma(L))}\sum_{g \in \Cent_G(u^{\gamma})} \trace(u^{\gamma} \vert \cE_g) \trace(g \vert \gamma(L)) \\
& = & \chi_{(u^{\gamma},\gamma(L))}
\end{eqnarray*}
where $u^{\gamma}$ denotes the result of raising $u$ to the power of the image of $\gamma$ in $\bZ_p^*$.
\end{proof}

\subsection{Carlsson's modules}
\label{sec:carlsson}

Give a $\bZ_q[G]$-module $\overline{M}$, we raised the question in \S\ref{intro:q1} of finding an ``even realization'' for it.  That is, the problem is to find $M \in \LMod(\KU_q[G])$ with $\pi_1 M = 0$ and $\pi_0 M = \overline{M}$.  In \S\ref{sec:four} we give an important class of positive examples, but let us here show that it is not always possible.

\begin{prop*}
If $G = (\bZ/2)^{\times 3}$, there are $\bZ_2[G]$-modules that cannot be realized in $K$-theory.
\end{prop*}

\begin{proof}
The counterexamples come by following a construction of Carlsson's \cite{Carlsson}.  Fix a free resolution $\overline{F}_{\bullet} \to \bZ_2$ of the trivial $\bZ_2[G]$-module, and let $W_n$ by the $n$th syzygy module of the resolution, i.e. $W_1$ is the kernel of $\overline{F}_0 \to \bZ_2$ and $W_n$ is the kernel of $\overline{F}_{n-1} \to \overline{F}_{n-2}$.  Then $\Ext^1(W_n,\bZ_2) \cong H^n(BG,\bZ_2)$.  For each $f \in H^n(BG,\bZ_2)$, let $\overline{M}(f)$ be the corresponding extension of $\bZ_2$ by $W_n$.

For $n > 0$, since $G$ has exponent $2$, the natural map $H^n(BG,\bZ_2) \to H^n(BG;\bF_2)$ is an injection --- it identifies $H^n(BG,\bZ_2)$ with the kernel of the first Steenrod square $\mathrm{Sq}^1$.  In the rest of the proof we will abuse notation and not distinguish between $f \in H^n(BG,\bZ_2)$ and its image in $H^n(BG,\bF_2)$.  

We claim that, if $\overline{M}(f)$ is realizable over $K$-theory, then $\mathrm{Sq}^3(f)$ must be divisible by $f$.   Indeed, suppose $M \in \LMod(\KU_2[G])$ has $\pi_0 M = \overline{M}(f)$ and $\pi_1 M = 0$.  The Adams spectral sequence \eqref{eq:E2st} for $[M,\KU_2]$ has $E_2^{st} = \Ext^s(\overline{M}(f),\bZ_2)$ for $t$ even and $E_2^{st} = 0$ for $t$ odd --- in particular $E_2= E_3$.  As a module over $H^*(BG,\bZ_2) = \Ext^*(\bZ_2,\bZ_2)$, the annihilator of $E_3^{st}$ is the principal ideal generated by $f$ --- indeed except in degree $0$, $\Ext^s(\overline{M}(f),\bZ_2)$ is given by the cokernel of multiplication-by-$f$ on $H^{s-n}(BG,\bZ_2)$.  The claim now follows from the compatibility of the module structure with the differentials in the Adams spectral sequence, and the fact that in the Atiyah-Hirzebruch spectral sequence for $H^*(BG,\bZ_2)$, the third differential is the Steenrod squaring operation $\mathrm{Sq}^3$ \cite[Prop. 2.4(d)]{Atiyah}.

It remains to find $n$ and $f \in H^n(BG,\bF_2)$ such that $\mathrm{Sq}^1(f) = 0$ and $\mathrm{Sq}^3(f)$ is not divisible by $f$.  Recall $H^*(BG,\bF_2)$ is a polynomial algebra in the generators of $H^1(BG,\bF_2)$ which we denote by $x,y,z$.  A direct computation shows we can take $n = 4$ and $f = x^4 + (x+y+z)xyz$.
\end{proof}

The construction is essentially identical to Carlsson's construction of a module that cannot be realized in a $G$-equivariant Moore space, but there the relevant annihilator must be stable under $\mathrm{Sq}^2$ rather than $\mathrm{Sq}^3$.  This stronger condition allows Carlsson to build a counterexample to the equivariant Moore space problem when $G = \bZ/2 \times \bZ/2$.  I do not know whether there are any $\bZ_2[\bZ/2 \times \bZ/2]$-modules that cannot be realized over $K$-theory.

\section{Study of $p$-permutation modules}
\label{sec:four}

This section is mostly self-contained.  We give a $K$-theoretic analog of $p$-permutation modules.

\subsection{Classical $p$-permutation modules}
Let us first recall the notion of $p$-permutation module \cite{Broue} in a form that makes the analogy easy to see.

Let $A$ be a commutative ring, and consider the category $\cF(G,A)$ whose objects are finite $G$-sets, whose morphisms from $X$ to $Y$ are given by matrices whose entries are chosen in $A$ and that are fixed by the action of $G$, and whose composition law is multiplication of matrices.  Put another way, the homomorphisms from $X$ to $Y$ are given by $G$-invariant functions $X \times Y \to A$ and the multiplication is given by convolution.

$\cF(G,A)$ is an additive category, where $X \oplus Y = X \amalg Y$.  Let $\cP(G,A)$ denotes its Karoubian completion, whose objects are given by pairs $(X,e_X)$, where $X \in \cF(G,A)$ and $e_X \in \mathrm{End}(X)$ is an idempotent.  For example when $G$ is trivial $\cF(G,A)$ is equivalent to the category of finite rank free $A$-modules, and $\cP(G,A)$ is equivalent to the category of finite rank projective $A$-modules.  In general we have a full embedding of $\cF(G,A)$ and $\cP(G,A)$ into the abelian category of $A[G]$-modules, which identifies $X \in \cF(G,A)$ with the permutation module $A[X]$.  When $A$ is a finite extension of $\bF_p$ or $\bZ_p$, this embedding identifies objects of $\cP(G,A)$ with what are usually called $p$-permutation  modules.  

\subsection{$K$-theoretic version}
\label{subsec:42}
If $X$ is a finite $G$-set, let $K_G(X)$ denote the Grothendieck group of virtual $G$-equivariant complex vector bundles over $X$.  When $X$ is a point we write this as $R(G)$.  In general $K_G(X) = \bigoplus_{x} R(G_x)$ where $x$ runs through $G$-orbit representatives in $X$ and $G_x$ denotes the stabilizer.  Let $\KF(G)$ denote the category whose objects are finite $G$-sets, and whose morphisms from $X$ to $Y$ are given by elements of $K_G(X \times Y)$.  The composition structure is given by convolution of vector bundles: if $\cE$ is a vector bundle over $X \times Y$ and $\cF$ is a vector bundle over $Y \times Z$, the convolution is given by
\begin{equation}
\label{eq:conv}
\left( \cF \ast \cE\right)_{(x,z)} = \bigoplus_{y \in Y} \cF_{(y,z)} \otimes \cE_{(x,y)}
\end{equation}
The endomorphism rings in this category have been considered by Lusztig.  In \cite[\S2.2(b),\S2.2(e)]{Lusztig1} it is shown that when tensored with $\bC$, they are semisimple and split as a product of endomorphism rings of $\cF(\Cent_G(g);\bC)$ where $g \in G$ runs through conjugacy class representatives.  More precisely, Lusztig shows
\begin{equation}
\label{eq:Lusztig}
K_G(X \times X) \otimes \bC \cong \prod_g \Fun(X^g \times X^g,\bC)^{\Cent_G(g)}
\end{equation}
For a commutative ring $A$,  let $\KF(G,A)$ denote the category with the same objects, and with morphisms given by $K_G(X \times Y) \otimes_{\bZ} A$ and the same composition law \eqref{eq:conv}.  Let $\KP(G,A)$ denote its Karoubi completion.

It can be deduced from \eqref{eq:Lusztig} that the isomorphism classes of indecomposable objects of $\KP(G,A)$ are parametrized by conjugacy classes of pairs $(g,L)$ where $g \in G$ and $L$ is a complex irreducible representation of $\Cent_G(g)$, or equivalently by isomorphism classes of irreducible objects in the category of conjugation-equivariant vector bundles on $G$ itself.  We seek a similar parametrization when $A = \bZ_q$.  We will prove the following

\begin{theorem}
\label{thm:pperm}
The set of indecomposable objects of $\KP(G,\bZ_q)$ is in natural bijection with the set of conjugacy classes of pairs $(g,M)$, where $g \in G$ has order prime to $p$ and $M$ is an indecomposable $p$-permutation module (in the usual sense) of the centralizer $\Cent_G(g)$
\end{theorem}

\subsection{Components of $\KP(G,\bZ_q)$}
\label{subsec:comp}
The representation ring $R(G)$ is reduced, torsion-free, and of Krull dimension 1 over $\bZ$.  When tensored with $\bZ_q$ it becomes disconnected.  Each connected component is a local ring, thus the primitive idempotents are in natural bijection with maximal ideals, which by \cite[Prop. 6.4]{Atiyah} are in natural bijection with conjugacy classes in $G$ of order prime-to-$p$.

If $C \subset G$ is such a conjugacy class, a formula for the corresponding idempotent $e_C$ , given in \cite{Bonnafe1} is 
\[
e_C = \frac{1}{|G|} \sum_{L \in \mathrm{Irr}_{\bC}(G)} \left( \sum_{g' \in \cS_{p'}(C)} \trace(g'^{-1} \vert L) \right) L
\]
where the inner sum is over those $g' \in G$ whose prime-to-$p$ part belongs to $C$.  Bonnaf\'e proves this formula with $\bZ_q$ replaced by a ring with sufficiently many $p$-power roots of $1$, but as $\cS_{p'}(C)$ is stable under the cyclotomic action of $\bZ_p^*$, the idempotent actually belongs to $R(G) \otimes \bZ_q$.  When $C = \{1\}$, $e_C = e_{\{1\}}$ is called the ``principal idempotent.''  Note that $(1 - e_\{1\})R(G) \otimes \bZ_q$ is exactly the Kuhn ideal of \eqref{eq:Kuhn}.

The categories $\KF(G,\bZ_q)$ and $\KP(G,\bZ_q)$ have an obvious $R(G) \otimes \bZ_q$-linear structure.  If $C \subset G$ is a conjugacy class of order prime to $p$, and $e_C$ is the corresponding primitive idempotent, then
\begin{enumerate}
\item Let $\KF_C(G,\bZ_q)$ denote the category with the same objects as $\KF(G,\bZ_q)$, whose morphisms are given by $e_C(K_G(X \times Y \otimes \bZ_q))$
\item Let $\KP_C(G,\bZ_q)$ denote the idempotent completion of $\KF_C(G,\bZ_q)$.  It is equivalent to the full subcategory of $\KP(G,\bZ_q)$ of pairs $(X,e_X)$ for which $e_X = e_X e_C$.  
\end{enumerate}

\subsection{The trace at $C$}
\label{subsec:traceC}
For a conjugacy class $C \subset G$ of exact order $m$ prime to $p$, fix a representative $c \in C$.  Let $X^c \times Y^c$ denote the $c$-fixed points of $X \times Y$.  If $\cE$ is a virtual $G$-equivariant vector bundle over $X \times Y$ and $(x,y) \in X^c \times Y^c$, the trace of $c$ on $\cE_{(x,y)}$ belongs to $\bZ(\sqrt[m]{1})\subset \bZ_q$.  This gives a map $K_G(X \times Y) \to \Fun(X^c \times Y^c,\bZ_q)^{\Cent_G(c)}$, and in particular
\[
\trace_C: (K_G(X \times Y) \otimes \bZ_q)e_C \to \Fun(X^c\times Y^c,\bZ_q)^{\Cent_G(c)}
\]
Thus $\trace_C$ induces a functor $\KP_C(G,\bZ_q) \to \cP(\Cent_G(c),\bZ_q)$, which we also denote by $\trace_C$.  Theorem \ref{thm:pperm} is a consequence of the following:

\begin{theorem*}
If $M \in \KP_C(G,\bZ_q)$ is indecomposable, so is $\trace_C(M)$, and this assignment gives a bijection on isomorphism classes of indecomposables.
\end{theorem*}

\subsection{The map $\mathrm{B}$}
\label{subsec:mapB}
To prove the Theorem of \S\ref{subsec:traceC}, we will in  \S\ref{subsec:BT} and \S\ref{subsec:relativeversion} reduce to the case $C = \{1\}$, by constructing an equivalence $\mathrm{B}$ making the triangle commute:
\[
\xymatrix{
\KP_C(G,\bZ_q) \ar[dd]_{\mathrm{B}} \ar[dr]^{\trace_C} \\
& \cP(\Cent_G(c),\bZ_q) \\
\KP_{\{1\}}(\Cent_G(c),\bZ_q) \ar[ur]_{\trace_{\{1\}}}
}
\]
The case of $C = \{1\}$ is treated in \S\ref{subsec:KandH}--\S\ref{subsec:Hausdorffness}.

\subsection{Bonnaf\'e's theorem}
\label{subsec:BT}
Let $e_C$ be as in \S\ref{subsec:comp}.  In \cite{Bonnafe1}, Bonnaf\'e shows that each of the blocks $(R(G) \otimes \bZ_q)e_C$ is isomorphic to the principal block $(R(\Cent_G(c)) \otimes \bZ_q)e_{\{1\}}$.  The isomorphism is via the ``$c$-translated restriction map''
\begin{equation}
\label{eq:bonnafe}
V \mapsto \frac{\trace(c\vert V)}{\dim(V)} \Res^G_{\Cent_G(c)} V \qquad \qquad \text{when $V$ is irreducible}
\end{equation}
Let us denote the map $(R(G) \otimes \bZ_q) e_C \to (R(\Cent_G(c)) \otimes \bZ_q) e_{\{1\}}$ obtained by first applying \eqref{eq:bonnafe} and then multiplying by the principal idempotent $e_1 \in R(\Cent_G(c))$, by $\beta_0$.  Thus, Bonnaf\'e proves $\beta_0$ is an isomorphism of rings.

\subsection{Relative version of Bonnaf\'e's theorem}
\label{subsec:relativeversion}
Let $X$ be a finite $G$-set, so that $K_G(X) \otimes \bZ_q$ is a $R(G) \otimes \bZ_q$-module.  We wish to relate $(K_G(X) \otimes \bZ_q)e_C$ to the fixed-point set $X^c$.

It is observed in \cite[\S 1.B]{Bonnafe1} that $R(\Cent_G(c))$ has a ring automorphism $t_c$ carrying each irreducible representation $L$ to $\frac{\trace(c\vert L)}{\dim(L)} L$.  This is compatible with a natural automorphism of $K_{\Cent_G(c)}(X^c)$ which we will also denote by $t_c$, sending each irreducible vector bundle $\cL$ to $\frac{\trace(c\vert \cL_y)}{\dim(\cL_y)}\cL$, where $y \in X^c$ is an arbitrary point in the support of $\cL$.

Define a map
\[
\beta_X:(K_G(X) \otimes \bZ_q)e_C \to (K_{\Cent_G(c)}(X^c) \otimes \bZ_q) e_{\{1\}} \qquad \beta_X(\cE) := e_{\{1\}} t_c(\cE \vert_{X^c})
\]

\begin{lemma*}
The map $\beta_X$ is an isomorphism.
\end{lemma*}

\begin{proof}
It suffices to prove this when $X = G/H$.  We will use the facts that $K_G(G/H) = R(H)$ and that \cite[\S 2.D]{Bonnafe1} the restriction map $R(G) \otimes \bZ_q \to R(H) \otimes \bZ_q$ carries $e_C$ to $\sum_D e_D$ where $D$ runs through $H$-conjugacy classes in $H \cap C$.

If $H \cap C$ is empty, then both the domain and codomain of $\beta_X$ are zero.  Indeed $(G/H)^c$ is empty when no conjugacy of $c$ belongs to $H$, and if $H \cap C$ is empty then the restriction of $e_C$ to $R(H) \otimes \bZ_q$ is zero.

So suppose $H \cap C$ is not empty, and without loss of generality that $c \in H$.  In this ace we have a map from $(G/H)^c$ to the set of $H$-conjugacy classes in $C \cap H$, by sending $gH$ to the class of $g^{-1} c g$.  Denote this conjugacy class by $\delta(gH)$, and let $\delta^{-1}(D)$ denote the fiber above the conjugacy class $D \subset H \cap C$.  Then $\Cent_G(c)$ acts transitively on $\delta^{-1}(D)$, and the stabilizer at $gH \in \delta^{-1}(D)$ is $\Cent_H(g^{-1} c g)$.  (The map $\delta$ is the connecting homomorphism for a long exact sequence in nonabelian cohomology of the group generated by $c$, see \cite[\S I.5.4-\S I.5.5]{Serre}.)  Just as we write $c$ for a representative of $C$, let us write $d$ for the representative $g^{-1} c g \in D \subset G \cap C$.  Thus we have a decomposition
\[
(K_{\Cent_G(c)}(X^c) \otimes \bZ_q)e_{\{1\}} = \bigoplus_D R(\Cent_H(d)) e_{\{1\}}
\]
Now the Lemma follows from the commutativity of the square
\[
\xymatrix{
(K_G(X) \otimes \bZ_q)e_C \ar[r]^{\beta_X} \ar[d] & K_{\Cent_G(c)}(X^c) e_{\{1\}} \ar[d] \\
\bigoplus_D (R(H) \otimes \bZ_q) e_D \ar[r]_{\bigoplus \beta_0} &  \bigoplus_D R(\Cent_H(d))e_{\{1\}}
}
\]
Here the lower arrow applies Bonnaf\'e's isomorphism $\beta_0$ to each summand, with $G$ replaced by $H$, and $(C,c)$ replaced by $(D,d)$.
\end{proof}

Now we define the equivalence $\mathrm{B}$ of \S\ref{subsec:mapB}. It suffices to define an equivalence $\mathrm{B}:\KF_C(G,\bZ_q) \to \KF_{\{1\}}$.  On objects we put $\mathrm{B}(X) = X^c$, and on him spaces we define
\[
\mathrm{B} = \beta_{X \times Y} : (K_G(X\times Y) \otimes \bZ_q)e_C \to (K_G(X^c \times Y^c) \otimes \bZ_q)e_{\{1\}}
\]
According to the Lemma, this is an equivalence of categories, and the commutativity of the triangle of \S\ref{subsec:mapB}.

\subsection{The $\bZ_q$-algebras $\cK$ and $\cH$}
\label{subsec:KandH} 
Let $Y$ be a finite $G$-set with the property that every other finite $G$-set admits a $G$-equivariant inclusion into $Y^{\amalg n}$ for some $n$.  For example, one could take $Y$ to be the set of all cosets $gG' \subset G$ where $G'$ runs through all subgroups.  We define two associative rings:
\begin{enumerate}
\item Let $\cK$ be the endomorphism ring of $Y$ in $\KF_{\{1\}}(G,\bZ_q)$,
\item Let $\cH$ be the endomorphism ring of $Y$ in $\cF(G,\bZ_q)$
\end{enumerate}
Since every object of either category is a summand of $Y^{\amalg n}$ for some $n$, the idempotent completion of the category of finite rank free $\cK$-modules, resp. $\cH$-modules, is the same as the Karoubi completion of the category of finite rank free $\cK$-modules, resp. $\cH$-modules.
Thus $\KP(G,\bZ_q)$ is equivalent to the category of finitely-generated projective $\cK$-modules and $\cP(G,\bZ_q)$ is equivalent to the category of finitely-generated projective $\cH$-modules.

\subsection{The augmentation ideal $J \subset \cK$}
As $\cK$ and $\cH$ are finite rank and free over the complete ring $\bZ_q$, they are both semiperfect (see \S\ref{subsec:cov}), thus the categories $\KP(G,\bZ_q)$ and $\cP(G,\bZ_q)$ have the Krull-Schmidt property.  In particular, each indecomposable is a summand of the generating object $Y$.

Let $J$ denote the kernel of the homomorphism $\cK \to \cH$.  We claim that if $J$ is \emph{Hausdorff}, in the sense that $\bigcap J^n = 0$, then the Theorem of \S\ref{subsec:traceC} holds when $C = \{1\}$.  Indeed if this is the case there can be no idempotent elements belonging to $J$, thus every indecomposable module is mapped to some nonzero $\cH$-module, and a standard argument shows that any idempotent of $\cH$ lifts to one of $\cK$.

We make one further reduction.  Note that $\cK = \cK_0 \otimes_{\bZ_p} \bZ_q$, where $\cK_0$ is the endomorphism ring of $Y$ in $\KF_{\{1\}}(G,\bZ_p)$.  Similarly $\cH = \cH_0 \otimes_{\bZ_p} \bZ_q$. The augmentation map $\cK \to \cH$ is induced by the augmentation map $\cK_0 \to \cH_0$ with kernel $J_0$.  As $\bZ_q/\bZ_p$ is flat, $J^n = J_0^n \otimes \bZ_q$.  To show that $J$ is Hausdorff it therefore suffices to show $J_0$ is Hausdorff, i.e. we may assume $q = p$.

\subsection{Hausdorffness of $J$}
\label{subsec:Hausdorffness}
We suppose $q = p$.  To prove the $J \subset \cK$ is Hausdorff it suffices to find any Hausdorff filtration of $\cK$
\[
\cK = \cK^{\geq 0} \supset \cK^{\geq 1} \supset \cK^{\geq 2} \supset \cdots
\]
such that $\cK^{\geq 1} = J$ and $\cK^{\geq s} \cK^{\geq t} \subset \cK^{\geq s+t}$.  In this section we explain how the Atiyah-Hirzebruch filtration provides such a $\cK^{\geq \bullet}$, completing the proof of Theorem \ref{thm:pperm}.

Let $(Y \times Y)_{hG}$ denote the Borel construction of the $G$-action on $Y \times Y$.  It is the disjoint union of classifying spaces for subgroups $G_{y_1} \times G_{y_2} \subset G$, where $(y_1,y_2)$ runs through orbit representatives of $G$ on $Y \times Y$.  It follows from \cite[Thm. 7.2]{Atiyah} that $\cK$ is naturally identified with $K((Y\times Y)_{hG}) \otimes_{\bZ} \bZ_p$.

For a CW complex $S = \bigcup S_k$, the Atiyah-Hirzebruch filtration of $K^0(S)$ is given by $F^{\geq k} K(S) = \ker(K(S) \to K(S_{k-1}))$, the set of virtual vector bundles that vanish over a $(k-1)$-skeleton.  As $K(S)$ is defined to be the inverse limit of the $K(S_k)$, this is a Hausdorff filtration.  The filtration is compatible with the tensor product of vector bundles: 
\begin{equation}
\label{eq:231}
\text{
for $\cE_1 \in K^{\geq s}(S)$ and $\cE_2 \in K^{\geq t}(S)$, the tensor product $\cE_1 \otimes \cE_2 \in K^{\geq s+t}(S)$.}
\end{equation}
See \cite[\S\S2--3]{Atiyah} for proofs and references.

Put $\cK^{\geq s} = F^{\geq s}K ((Y \times Y)_{hG}) \otimes \bZ_p$.  To complete the proof we must show that 
\begin{equation}
\label{eq:232}
\text{
for $\cE_1\in \cK^{\geq s}$ and $\cE_2 \in \cK^{\geq t}$, the convolution $\cE_1 \ast \cE_2 \in \cK^{\geq s+t}$.  
}
\end{equation}
This condition can be checked on stalks.  By definition, the stalk of $\cE_1 \ast \cE_2$ at $(y_1,y_2)$ is the sum
$
\bigoplus_{y \in Y} (\cE_1)_{y_1,y} \otimes (\cE_2)_{y,y_2}
$, so that \eqref{eq:232} is consequence of \eqref{eq:231}.

\appendix
\pagestyle{app}

\newcommand{\ein}{\mathrm{E}_\infty}
\newcommand{\rw}[2]{\mathrm{Rep}( #1, #2)}
\newcommand{\rwk}[1]{\mathrm{Rep}( #1, \Kp)}
\newcommand{\md}{\mathrm{LMod}}
\newcommand{\Kp}{\mathbf{K}_p}
\newcommand{\coh}{\mathrm{Coh}}
\newcommand{\supp}{\mathrm{Supp}}
\newcommand{\fun}{\mathrm{Fun}}
\newcommand{\spec}{\mathrm{Spec}}
\newtheorem*{notation}{Notation}
\newcommand{\rwkp}{\rwk{\mathbf{Z}/p}}

\section{Representations in finite $\Kp$-modules, by Akhil Mathew}

\subsection{Generalities}
Let $G$ be a finite group. In this appendix, we consider certain aspects of
the representation theory of $G$ in \emph{perfect} $\Kp$-module spectra. 
We raise seemingly natural questions, which we are only able to answer
in certain specific cases. 
Recall that if $R$ is an $\ein$-ring, then we let $\md^\omega(R)$ denote the
$\infty$-category of perfect $R$-modules. 

\begin{definitionAM} 
Let $R$ be an $\ein$-ring spectrum and let $G$ be a finite group.
We define $\rw{G}{R} = \mathrm{Fun}(BG, \md^\omega(R))$ to be the
$\infty$-category of perfect $R$-modules equipped with a $G$-action.
\end{definitionAM} 

We observe that $\rw{G}{R}$ is naturally a symmetric monoidal, stable
$\infty$-category with the $R$-linear tensor product, which is exact in both
variables. We may think of it as the analog of the category of
\emph{finite-dimensional} representations of a group over a field. 
As in that case, we have a natural fully faithful, symmetric monoidal functor
\( \rw{G}{R} \to \md( R[G]),   \)
which prolongs to a symmetric monoidal, colimit-preserving functor 
\begin{equation} \label{indfunctor} \mathrm{Ind} ( \rw{G}{R}) \to \md( R[G]).
\end{equation}
In general, \eqref{indfunctor} is not an equivalence,  unless the order of $G$
is invertible in $R$ or in certain special situations. 
Nonetheless, the $\infty$-category $\rw{G}{R}$ is often of interest, and one
can ask several basic questions of it. 

The first, and perhaps most basic, question is whether one has a reasonable set
of building 
blocks for $\rw{G}{R}$, i.e., that there are somehow not too many
$G$-actions on perfect $G$-modules.
\begin{questionAM} 
Is there a finite set of objects in $\rw{G}{R}$ that generate it as a
thick subcategory?
\end{questionAM} 

Such a result  for the $\ein$-ring $R$ would be analogous to the fact that a finite group has only
finitely many irreducible representations over a given field.
There are certain natural objects that one can 
always construct. 
Given any finite $G$-set $S$, consider $R[S] 
\stackrel{\mathrm{def}}{=} R \wedge \Sigma^\infty_+ S
\in \rw{G}{R}$. When $S$ is the one-point $G$-set, we will sometimes write
simply $R$ for the associated (unit) object in $\rw{G}{R}$ when further confusion is unlikely to arise. 

\begin{questionAM} 
\label{generateperm}
Do the objects $R[S] \in \rw{G}{R}$, for $S$ ranging over the finite 
$G$-sets $S$, generate $\rw{G}{R}$ as a thick subcategory?
\end{questionAM}

The main result of the appendix is that the answer to Question \ref{generateperm} is
affirmative in the case of $R = \Kp$ and $G = \mathbf{Z}/p$.

\subsection{Examples}

To begin with, we include some results and counterexamples to show that
Question \ref{generateperm} is  nontrivial in general, even for discrete
$\ein$-rings. 

\begin{propositionAM} 
\label{rationalcase}
Suppose $R$ is an $\ein$-ring
and $G$ is a finite group. Suppose $|G| $ is invertible in $\pi_0(R)$. Then
$\rw{G}{R}$ is generated as a thick subcategory by 
$R[G]$.
\end{propositionAM} 
\begin{proof} 
In this case, we consider the map $\phi: R[G] \to R$ coming from the map of
$G$-sets $G \to \ast$  and its dual $\psi: R \to R[G]$. Then $ \frac{1}{|G|}
\psi \circ \phi$ is an idempotent endomorphism of $R[G]$ whose image is
identified with the trivial representation $R$. 
If $M \in \rw{G}{R}$, then $M \simeq M \otimes R$ is thus a retract of $M
\otimes R[G]$. However, the projection formula implies that $M
\otimes R[G]$ is the induced object $\mathrm{Ind}_1^G \mathrm{Res}^G_1 (M)$,
which clearly belongs to the thick subcategory generated by 
$\mathrm{Ind}_1^G \mathrm{Res}^G_1 (R) \simeq R[G]$. Taking retracts, it
follows that $M$ belongs to the thick subcategory generated by $R[G]$.
\end{proof}

Next, we give a counterexample in modular characteristic.
\begin{exampleAM} 
Let $R $ denote the commutative ring $
\mathbf{F}_2[\epsilon]/\epsilon^2$ (identified with a discrete
$\ein$-ring) and $G = \mathbf{Z}/2$. 
Then, in $\rw{\mathbf{Z}/2}{R}$, we can construct an invertible object $\mathcal{L}$: namely, we take
a free $R$-module of rank one, but with $\mathbf{Z}/2$-action given by
multiplication by $1 + \epsilon$. We claim that $\mathcal{L}$ does not belong to the
thick subcategory generated by the permutation modules. 

Indeed, since $ 2 = 0$ in $R$, the permutation module $R^{ \mathbf{Z}/2}$
belongs to the thick subcategory generated by the trivial representation
$R$. It
suffices to show that $\mathcal{L}$ does not belong to the thick subcategory
generated by $R$. Since $\mathcal{L}$ is  invertible (with inverse
$\mathcal{L}$ again), it is equivalent to showing that $R$ does not
belong to the thick subcategory generated by $\mathcal{L}$. 

However, we have an exact functor
\[ \rw{\mathbf{Z}/2}{R} \to \rw{\mathbf{Z}/2}{\mathbf{F}_2},  \]
given by restriction of scalars. 
This functor carries $\mathcal{L}$ to the permutation
object $\mathbf{F}_2[\mathbf{Z}/2] \in \rw{\mathbf{Z}/2}{\mathbf{F}_2}$. It
carries the unit of $\rw{\mathbf{Z}/2}{R}$ to the sum of two copies of the
unit in $\rw{\mathbf{Z}/2}{\mathbf{F}_2}$. 
However, the thick subcategory generated in $\rw{\mathbf{Z}/2}{\mathbf{F}_2}$
by $\mathbf{F}_2[{\mathbf{Z}/2}]$ does not contain the unit. This implies that
the thick subcategory $\mathcal{L}$ generates cannot contain the unit
$R$ in $\rw{\mathbf{Z}/2}{R}$. 
\end{exampleAM}

Nonetheless, we prove that if the ring is better behaved this phenomenon does
not occur. 
\begin{theoremAM} 
\label{generatebypermutation_regular}
Let $R$ be a regular noetherian ring of finite Krull dimension and let $G$ be a finite group.
Then the $\infty$-category $\rw{G}{R}$
is generated as a thick subcategory by the permutation modules $\{R[{G/H}]\}|_{H
\subset G}$.
\end{theoremAM}

Our strategy is to reduce to the case of a $p$-torsion object (for some $p$)
and then to prove a more precise result via a bit of commutative algebra. 
In the rest of this subsection, when we want to work with 
discrete modules over a discrete noetherian ring, we will use the subscript $0$
to avoid confusion.

\begin{definitionAM} 
Let $R$ be a noetherian ring and $M_0$ a finitely generated (discrete) $R$-module. Let
$\mathfrak{p}$ be a minimal prime ideal in $\supp M_0$.  In this case,
$(M_0)_{\mathfrak{p}}$ is a module over $R_{\mathfrak{p}}$ whose support is
concentrated at the closed point $\mathfrak{p}$ and is therefore of finite
length. We let $n_{\mathfrak{p}}(M_0) =
\ell_{R_{\mathfrak{p}}}((M_0)_{\mathfrak{p}})$. 
\end{definitionAM}

\begin{lemmaAM} 
\label{fixedpointnonzero}
Let $R$ be a noetherian ring and let $M_0$ be a discrete
$R$-module which is finitely generated and $p$-power torsion. Suppose given an $R$-linear action of a
$p$-group $G$ on $M_0$. Then the support of $M_0^{G}$
is equal to that of $M_0$.
\end{lemmaAM} 
\begin{proof} 
Suppose $M_0$ is annihilated by $p^n$.
The support of $M_0$ is, by definition, those prime ideals $\mathfrak{p}
\subset R$ such that $(M_0)_{\mathfrak{p}} \neq 0$. 
So it suffices to show that if $(M_0)_{\mathfrak{p}} \neq 0$, then
$(M_0^{G})_{\mathfrak{p}} \neq 0$ too. But 
\[ (M_0^{G})_{\mathfrak{p}} = 
(M_0)_{\mathfrak{p}}^{G} \neq 0
\]
because $(M_0)_{\mathfrak{p}}$ is a nonzero $\mathbf{Z}/p^n$-module with a
$G$-action, and any such has nonzero fixed points.
\end{proof} 

\begin{corollaryAM} 
\label{lengthgoesdown}
Hypotheses as in Lemma \ref{fixedpointnonzero}, for any minimal prime ideal $\mathfrak{p} \in \supp
M$, we have $n_{\mathfrak{p}}(M_0/M_0^G) < n_{\mathfrak{p}}(M_0)$.
\end{corollaryAM}

\begin{definitionAM} 
Let $R$ be a noetherian ring. 
Let $\coh(R)$
denote the $\infty$-category of $R$-modules $M$ such that: 
\begin{enumerate}
\item $\pi_i(M) \neq 0$ for $i \gg 0$ or $i \ll 0$. 
\item For each $i$, $\pi_i(M)$ is a finitely generated $R$-module.
\end{enumerate}
Given a prime number $p$, 
let $\coh^p(R)$ denote the subcategory of $\coh(R)$ spanned by those $M \in
\coh(R)$ such that, in addition, each of the homotopy groups $\pi_i(M)$ is $p$-power torsion. 
\end{definitionAM}

In general, $\coh(R) \subset \md(R)$ is a thick subcategory, but it is not
closed under tensor products. However, if $R$ is a \emph{regular} noetherian
ring of finite Krull dimension, then $\coh(R)$ is equal to
$\md^\omega(R)$ as a subcategory of $\md(R)$.

\begin{propositionAM} 
\label{reponptorsion}
Let $R$ be a noetherian ring  and let $G$ be a finite $p$-group.
Then the $\infty$-category $\fun(BG, \coh^p(R))$ is generated 
as a stable subcategory by the objects $A$ for $A \in
\coh^p(R)$ (given trivial $G$-action). 
\end{propositionAM} 

\begin{proof} 
Let $\mathcal{C} \subset \fun(BG, \coh^p(R))$ be the \emph{stable} subcategory generated
by the trivial actions on objects in $\coh^p(R)$. 

Suppose $M \in 
\fun(BG, \coh^p(R))$. 
To show that $M$ belongs to $\mathcal{C}$, we will use noetherian induction on
$\supp \pi_*(M) \subset \spec R$. 
If $\pi_*(M) =0$, there is nothing to prove.
We make the following inductive hypothesis for 
$Z \subset \spec R$ a  given closed subset.
 If $M'$ is an $R$-module with
$G$-action, all of whose homotopy groups 
are supported on proper subsets of $Z$, we will assume that  then $M'$ belongs
to $\mathcal{C}$.

Now suppose $M \in \fun(BG, \coh^p(R))$ with $\supp \pi_*(M) \subset Z$; we want
to show that $M \in \mathcal{C}$.
Since we can write $M$ as an iterated extension of
objects in $\fun(BG, \coh^p(R))$ each of which is a shift of a discrete $R$-module with
$G$-action, we may assume that $M$ itself arises from a $p$-power
torsion discrete $R$-module
$M_0 = \pi_0(M)$ with $G$-action; we will (by a slight abuse of notation) identify
$M_0$ and $M$. 
Moreover, $\supp M_0 \subset Z$. 

Now let $\mathfrak{p}_1, \dots, \mathfrak{p}_n$ be the minimal primes of $Z$. 
Using another induction, we may assume that any $N \in \fun(BG, \coh^p(R))$ such
that $N$ is discrete with $\supp \pi_0(N) \subset Z$ and $\sum_{i=1}^n
n_{\mathfrak{p}_i}(\pi_0(N))  < 
\sum_{i=1}^n
n_{\mathfrak{p}_i}(\pi_0(M))$ belongs to $\mathcal{C}$.

We know that the fixed point module $M_0^{G} \subset M_0$, considered as an object of $\coh^p(R)$
with trivial $G$-action, belongs to the thick subcategory (it is one of the
generators),
and it maps to $M$. Now we can
consider the quotient $M_0/M_0^G$. By Corollary \ref{lengthgoesdown} and the second inductive hypothesis, we 
know that $M_0/M_0^G \in \mathcal{C}$.
Putting everything together, we get that $M_0 \in \mathcal{C}$. 
\end{proof}

\begin{corollaryAM} 
Let $R$ be a noetherian ring  and let $G$ be a finite group.
Let $G_p \subset G$ be a $p$-Sylow subgroup of $G$.
Then the $\infty$-category $\fun(BG, \coh^p(R))$ is generated 
as a thick subcategory by the objects $A \wedge (G/G_p)_+$ for $A \in
\coh^p(R)$. 
\label{reponptorsion2}
\end{corollaryAM} 
\begin{proof} 
Fix $M \in \fun(BG, \coh^p(R))$.
One sees easily, as in the proof of Proposition \ref{rationalcase}, that $M$ is a retract of $\mathrm{Ind}_{G_p}^G
\mathrm{Res}_{G_p}^G (M) = M \otimes_R R[G/G_p]$, 
as the composite
\[ R \stackrel{\mathrm{tr}}{\to}R[G/G_p]  \to R, \]
for the second map induced by the projection $G/G_p \to \ast$ and the first map
its dual (the \emph{transfer}) is given by multiplication by $[G: G_p]$.
So, one reduces to the case where $G$ is a $p$-group.  
In this case, one can use Proposition \ref{reponptorsion}.
\end{proof}

\begin{proof}[Proof of Theorem \ref{generatebypermutation_regular}] 
Let $\mathcal{D} \subset \rw{G}{R}$ denote the thick subcategory generated by
the permutation objects. 

Suppose $M \in \rw{G}{R}$. We want to prove that $M \in \mathcal{D}$. 
Without loss of generality, we can assume that $M$ is discrete. We let $M_0 =
\pi_0(M)$ be the associated discrete $R$-module with $G$-action. 
In this case, we know that the $G$-representation $M_0 \otimes \mathbf{Q}$ is a direct summand of an
induced object $\mathrm{Ind}_1^G \mathrm{Res}_1^G( M_0\otimes \mathbf{Q}) $. 
In particular, we have a (discrete) finitely generated $(R \otimes
\mathbf{Q})[G]$-module $V_0$ such that 
we have an equivalence of $(R \otimes \mathbf{Q})[G]$-modules
\[ M_0 \otimes \mathbf{Q}  \oplus V_0 \simeq  \mathrm{Ind}_1^G
\mathrm{Res}_1^G( M_0\otimes \mathbf{Q}). \]
Now choose a $G$-stable $R$-submodule $W_0 \subset V_0$ such that $W_0 \otimes
\mathbf{Q} =V_0$ and such that $W_0$ is a finitely generated $R$-module.

Now $N_0 \stackrel{\mathrm{def}}{=}M_0 \oplus W_0$ defines an object in $\rw{G}{R}$; by definition
of a thick subcategory, it suffices to
show that this object belongs to $\mathcal{D}$. 
However, we know that the $(R \otimes \mathbf{Q})[G]$-module $N_0 \otimes
\mathbf{Q} $ is a induced module $ \mathrm{Ind}_1^G( M_0 \otimes \mathbf{Q})$
which contains the induced $R[G]$-module $\mathrm{Ind}_1^G(M_0/\mathrm{tors})$. 

Multiplying by a highly divisible integer if necessary, we may assume that
we have an inclusion $N_0 
 \subset \mathrm{Ind}_1^G(M_0/\mathrm{tors})$. 
Since 
$ \mathrm{Ind}_1^G(M_0/\mathrm{tors}) \in \mathcal{D}$, it suffices to show that the quotient
(or rather, the object in $\rw{G}{R}$ defined by it) $X =
\mathrm{Ind}_1^G(M_0/\mathrm{tors})/
N_0$ belongs to $\mathcal{D}$. This object $X$, however, is torsion. 
It therefore decomposes as a \emph{finite} direct sum, 
\[ X = \bigoplus_{p} X_p,  \]
where each $X_p \in \rw{G}{R}$ is $p$-power torsion. We can now use
Corollary \ref{reponptorsion2} to conclude that each $X_p\in \mathcal{D}$, so that $X
\in \mathcal{D}$ too.

\end{proof}

\subsection{The main result}

The main theorem of this appendix states: \begin{theoremAM} 
\label{KUZ/pgenerators}
Let $\Kp$ denote $p$-adically completed complex $K$-theory.
The smallest stable subcategory 
of $\rwk{\mathbf{Z}/p}$ generated by $\Kp$ and $\Kp[\mathbf{Z}/p]$ is all of
$\rwk{\mathbf{Z}/p}$.
\end{theoremAM} 

We note Theorem \ref{KUZ/pgenerators} does not even require thick subcategories.


Our proof of Theorem \ref{KUZ/pgenerators} is somewhat ad hoc: it relies on an
analysis of the possible homotopy fixed point spectral sequences converging to
$\pi_*(M^{h\mathbf{Z}/p})$ for $M \in \rwk{\mathbf{Z}/p}$, together with the
(known)
purely algebraic classification of torsion-free representations of $\mathbf{Z}/p$ over the $p$-adics
$\mathbf{Z}_p$.

We will break through proof of 
Theorem \ref{KUZ/pgenerators} into a series of steps. First, we begin by recalling the
purely algebraic ingredient that we need.

\begin{theoremAM}[{\cite[Th. 2.6]{HeRe62}}] 
\label{algclassZ/p}
Let $V$ be a (discrete) module over the group ring $\mathbf{Z}_p[\mathbf{Z}/p]$
whose underlying $\mathbf{Z}_p$-module is finitely generated and free. Then $V$
is a finite direct sum of the $\mathbf{Z}_p[\mathbf{Z}/p]$-modules $V_1, V_2,
V_3$ where $V_1 = \mathbf{Z}_p[\mathbf{Z}/p]$ is the induced representation,
$V_2 = \mathbf{Z}_p$ is the trivial representation, and $V_3 =
V_1/V_1^{\mathbf{Z}/p}$ is the natural quotient $V_1/V_2$.
\end{theoremAM}

Let $M \in \rwk{\mathbf{Z}/p}$ be such that $\pi_*(M)$ is torsion-free. Then
$\pi_*(M)$ is a $\mathbf{Z}_p[\mathbf{Z}/p]$-module. 
By 
Theorem \ref{algclassZ/p}, $\pi_0(M)$ and $\pi_1(M)$ are direct sums of the indecomposable building
blocks  $V_1, V_2, V_3$. 
Our goal is to analyze the possible homotopy fixed point spectral sequences for
$\pi_*(M^{h \mathbf{Z}/p})$. 

To begin with, we write down the $\mathbf{Z}/p$-cohomology 
of these objects as modules over $R_* \stackrel{\mathrm{def}}{=} H^*(\mathbf{Z}/p; \mathbf{Z}_p) \simeq
\mathbf{Z}_p[x_2]/(px_2)$.
\begin{propositionAM} 
\label{cohomologycalc}
As graded modules over $R_*$, one has: 
\begin{itemize}
\item  $H^*(\mathbf{Z}/p; V_1) \simeq \mathbf{Z}_p \simeq R_*/(x_2)$
(it begins in degree zero).
\item $H^*(\mathbf{Z}/p; V_2) \simeq R_*$.
\item $H^*(\mathbf{Z}/p; V_3) \simeq R_*\left\{y_1\right\}/(py_1)$ with $|y_1|
= 1$.
\end{itemize}
\end{propositionAM}

Proposition \ref{cohomologycalc} is well-known, and the third calculation 
follows from the first two. 

Consider now the HFPSS for $M \in \rwk{\mathbf{Z}/p}$, which runs 
\begin{equation} \label{HFPSS} H^s( \mathbf{Z}/p; \pi_t(M)) \implies \pi_{t-s} (M
^{h\mathbf{Z}/p}).  \end{equation}
The $E_2$-term is 2-periodic because $\Kp$ is 2-periodic, and it is a module
over the spectral sequence for $\Kp$ with trivial action.
As a result, we obtain: 
\begin{propositionAM} 
The differentials in the HFPSS \eqref{HFPSS} are $R_*$-linear. 
\end{propositionAM} 
\begin{proof} 
This follows from the fact that \eqref{HFPSS} is a module over the spectral
sequence for $\Kp$ with trivial action, which degenerates at $E_2$.
\end{proof}

We now make a key observation about the structure of the HFPSS in case
$\pi_*(M)$ is torsion-free.
We will regard the $E_r^{\ast, \ast}$ pages of the HFPSS as graded
$R_*$-modules, where the grading we  use on $E_r^{\ast, \ast}$ is the
cohomological (vertical) one. 
Before doing so, we need to recall a basic construction for producing permanent
cycles in the HFPSS.

\begin{consAM}
Let $ M \in \rwk{\mathbf{Z}/p}$. 
Then there is a natural map, called the \emph{norm} map 
\[ M \to M^{h \mathbf{Z}/p} , \]
such that the composite 
$M \to M^{h\mathbf{Z}/p} \to M$ is given by multiplication by $\Sigma_{g \in
\mathbf{Z}/p} g$.  
\end{consAM}

In particular, it follows that any element in $\pi_*(M)^{\mathbf{Z}/p}$ (which
is the $E_2^{0, \ast}$ of the HFPSS) that is a norm from $\pi_*(M)$ is a
permanent cycle in the HFPSS, as it lifts to an element of $\pi_*(M^{h
\mathbf{Z}/p})$.

\begin{propositionAM} 
\label{analyszeHFPSS}
Consider the HFPSS \eqref{HFPSS} for $M \in \rwk{\mathbf{Z}/p}$ with 
$\pi_*(M)$ torsion-free. View the $E_r$-term as a graded module over $R_*  =
H^*(\mathbf{Z}/p; \mathbf{Z}_p)$. 
Then 
one has a (graded) decomposition of 
\[ E_r \simeq  E_r^{\mathrm{tors}} + E_r^{\mathrm{free}}  + P_r,  \]
where: 
\begin{enumerate}
\item $E_r^{\mathrm{tors}} $ is the submodule of elements which are
$x_2$-power torsion. 
\item 
$E_r^{\mathrm{free}}$ is a  sum of copies of $R_*$,  and $R_{\ast - 1}/(p)$ (as a
graded module; the generators are in degree zero and one, respectively).
\item $P_r$ is a direct sum of copies of $\mathbf{Z}_p$ concentrated in degree
zero.
\end{enumerate}
Moreover, one has: 
\begin{enumerate}
\item  $E_r^{\mathrm{tors}} $ is concentrated in degrees $\leq r-1$ and is
generated by elements in degrees zero and one.
\item The differential $d_r$  annihilates $P_r$ and 
$E_r^{\mathrm{tors}} $; moreover, 
$E_r^{\mathrm{tors}} $ injects into the next page.
\end{enumerate}
\end{propositionAM} 
\begin{proof} 
This is a direct induction on $r$. At $r = 2$, it is clear in view of 
Proposition \ref{cohomologycalc}: the copies of $R_*$ come from copies of $V_2$, and the
copies of $R_{\ast - 1}/p$ come from copies of $V_3$.
Note that any contribution from an induced representation $V_1$ in $E_2^{0, t}$
must survive to $E_\infty$ (i.e., support no differentials) in view of the norm
map, so we can safely put any such contribution into $P$. 

Suppose we know that the result is true at stage $r$. 
Then, since $d_r$ must 
preserve $E_r^{\mathrm{tors}}$ and since it raises degrees by $r$, it follows
that it acts by zero on $E_r^{\mathrm{tors}}$. 
In particular, $d_r$ does not interact with $E_r^{\mathrm{tors}}$ for reasons
of grading.

Choose a minimal system of generators of degrees zero and one in
$E_r^{\mathrm{free}}$ as a module over $R_*$.  These generators are carried by
$d_r$ to elements of $E_r$ in degrees $r, r+1$, which is above the maximal
degree in $E_r^{\mathrm{tors}}$. It follows that $E_{r+1}^{\mathrm{tors}}$, as the
homology, is $E_r^{\mathrm{tors}}$ plus possibly additional copies of $R_*/x_2^k$ and
$R_*\left\{y_1\right\}/x_2^k y_1$, where $k$ is such that
$E_{r+1}^{\mathrm{tors}}$ has no terms above $r$. 

Finally, if any element $x
\in E_2^{0,t} = (\pi_t M)^{\mathbf{Z}/p}$ in a
summand of $R_*$ supports a differential, we note that $px$ is a permanent
cycle in view of the \emph{norm} map $M \to M^{h\mathbf{Z}/p}$, whose
composition with the natural map $M^{h\mathbf{Z}/p} \to M$ is multiplication
by $\sum_{g \in \mathbf{Z}/p} g$. So, in this case we place the remaining $\mathbf{Z}_p$ summand into $P$. 
This gives the desired inductive step. 
\end{proof}

We are now ready to complete the proof of the main result in a series of lemmas. 

\begin{notation} 
For the rest of the appendix, we will let $\mathcal{C} \subset \rwkp$ denote
the stable subcategory generated by the permutation modules $\Kp,
\Kp[\mathbf{Z}/p]$.
\end{notation}

\begin{lemmaAM} 
\label{firstlm}
Suppose $M \in \rwkp$ and:
\begin{enumerate}
\item $\pi_1(M)  = 0$.
\item The $\mathbf{Z}_p[\mathbf{Z}/p]$-module $\pi_0(M)$ is a sum of copies of
$V_1, V_2$.
\end{enumerate}
Then $M$ is equivalent to a direct sum of copies of $\Kp[\mathbf{Z}/p]$ and 
$\Kp$ (and in particular belongs to $\mathcal{C}$).
\end{lemmaAM}
\begin{proof} 
The hypotheses imply that there is no room for differentials in the homotopy
fixed point spectral sequence, so the map $\pi_0(M^{h\mathbf{Z}/p}) \to
\pi_0(M)^{\mathbf{Z}/p}$ is surjective. Choose a $\mathbf{Z}_p$-basis $x_1,
\dots, x_n$ for $\pi_0(M)^{\mathbf{Z}/p}$ and a
$\mathbf{Z}_p[\mathbf{Z}/p]$-basis $y_1, \dots, y_m$ for a complement (which
is  $\mathbf{Z}_p[\mathbf{Z}/p]$-free) to
$\pi_0(M)^{\mathbf{Z}/p}$ in $\pi_0(M)$.
The $\left\{y_i\right\}$ each determine maps $\psi_i: \Kp[\mathbf{Z}/p] \to M$
carrying $1 \mapsto y_i$ on $\pi_0$. Choose a lift $\widetilde{x}_j \in
\pi_0(M^{h\mathbf{Z}/p})$ for each $x_j$ to obtain maps  $\phi_j: \Kp \to M$
in $\rwkp$. Then the direct sum of the $\psi_i$ and the $\phi_j$ determines an
equivalence
\[ \bigoplus_{i = 1}^n \Kp[\mathbf{Z}/p] \oplus \bigoplus_{j = 1}^m \Kp  \simeq M \]
in the $\infty$-category $\rwkp$, by inspection of homotopy groups.
\end{proof} 

\begin{lemmaAM} 
Suppose $M \in \rwkp$ and:
\begin{enumerate}
\item $\pi_1(M)  = 0$.
\item The $\mathbf{Z}_p[\mathbf{Z}/p]$-module $\pi_0(M)$ is a sum of copies of
$V_1, V_3$.
\end{enumerate}
Then $M \in \mathcal{C}$.
\label{somethingsC'} 
\end{lemmaAM} 
\begin{proof} 
Choose a decomposition $\pi_0(M) \simeq \bigoplus_1^m V_1 \oplus \bigoplus_1^n
V_3$.
For each summand $V_1 \subset \pi_0(M)$, choose a map $\phi: \Kp[\mathbf{Z}/p]
\to M$ realizing that summand on homotopy. For each summand $V_3 \subset
\pi_0(M)$, choose a map $\psi: \Kp[\mathbf{Z}/p] \to M$ realizing the surjection
$V_1 \twoheadrightarrow V_3 $ composed with the inclusion $V_3 \subset M$ on
homotopy. Collecting these together, we obtain a map $\bigoplus_1^{n+m}
\Kp[\mathbf{Z}/p] \to M$ whose fiber $F$ has the property that its homotopy groups
are direct sums of $V_1$ (trivial modules). 
By Lemma \ref{firstlm}, we may now conclude that 
$F \in \mathcal{C}$, so $M \in \mathcal{C}$ too. 
\end{proof}

We now come to a key technical lemma that relies on Proposition \ref{analyszeHFPSS}.
\begin{lemmaAM} 
\label{pi1zero}
Suppose $M \in \rwkp$ and:
\begin{enumerate}
\item $\pi_1(M) = 0$. 
\item $\pi_1(M^{h\mathbf{Z}/p}) = 0$.
\item $\pi_0(M)$ is torsion-free.
\end{enumerate}
Then the map $\pi_0(M^{h \mathbf{Z}/p}) \to \pi_0(M)^{\mathbf{Z}/p}$ is
surjective.
\end{lemmaAM} 
\begin{proof} 
Consider the HFPSS for $M$. In this case, no differentials can leave $E_2^{0,0}
= (\pi_0 M)^{\mathbf{Z}/p}$ because any such would force a nontrivial
$\pi_{-1}$ in $M^{h\mathbf{Z}/p}$: by the analysis of the HFPSS in Proposition \ref{analyszeHFPSS}, a
nontrivial differential $d_r$ out of $E_r^{0, 0}$ would create elements in
$E_{r+1}^{\mathrm{tors}}$ with $t - s = -1$ that would then have to be
permanent cycles, leading to nontrivial contributions in $\pi_{-1}$.
This implies the surjectivity statement of the lemma.
\end{proof}

\begin{lemmaAM} 
\label{lotsofthingsmeanC'}
Suppose $M \in \rwkp$ and:
\begin{enumerate}
\item $\pi_1(M) = 0$. 
\item $\pi_1(M^{h\mathbf{Z}/p}) = 0$.
\item $\pi_0(M)$ is torsion-free.
Then $M \in \mathcal{C}$.
\end{enumerate}
\end{lemmaAM} 
\begin{proof} 
Choose a basis $x_1, \dots, x_n$ of 
$\pi_0(M)^{h\mathbf{Z}/p}$ as a $\mathbf{Z}_p$-module; by Lemma \ref{pi1zero}, we can find a map 
\[ \phi: \bigoplus_{i=1}^n \Kp \to M \]
which in homotopy induces the map $\mathbf{Z}_p^n \to \pi_0(M)$ corresponding
to the vectors $\left\{x_i\right\}$. 
Form the cofiber $C$ of $\phi$; it suffices to show 
that $C \in \mathcal{C}$. 

However, we see that $\pi_1(C)  = 0$ and that $\pi_0(C)$ is a sum of copies of
$V_1, V_3$ as a $\mathbf{Z}_p[\mathbf{Z}/p]$-module. By Lemma \ref{somethingsC'}, $C$ belongs
to the stable subcategory desired. 
\end{proof} 

\begin{proof}[Proof of Theorem \ref{KUZ/pgenerators}]
Given $M \in \rwk{\mathbf{Z}/p}$. We want to show that $M$ belongs to the
stable subcategory $\mathcal{C}$. 
First, choose a system of generators $x_1, \dots, x_n$ for $\pi_0(M) \oplus
\pi_1(M)$. We obtain a map 
\[ \bigoplus \Kp[\mathbf{Z}/p] \oplus \bigoplus \Sigma \Kp[\mathbf{Z}/p] \to M, \]
which is surjective on $\pi_*$.
Let $N$ be the fiber of this map. Clearly, it suffices to show that $N \in
\mathcal{C}$. The long exact sequence in homotopy shows that $\pi_*(N)$ is
torsion-free as a $\mathbf{Z}_p$-module. 

Next, we  choose a system of generators for $\pi_1(N)$, and obtain a map
\[ \bigoplus \Sigma \Kp[\mathbf{Z}/p] \to N,  \]
which is a surjection on $\pi_1$. Let $C$ be the cofiber of this map, in turn.
It suffices to show that $C \in \mathcal{C}$. However, we note that $\pi_1(C)
=0 $ (by the long exact sequence) and $\pi_0(C)$ is $\mathbf{Z}_p$-torsion-free. 

Finally, we choose a system of generators for $\pi_1(C^{h \mathbb{Z}/p})$, and
obtain a map
\[ \bigoplus \Sigma \Kp \to C,  \]
which induces a surjection in $\pi_1( ( \cdot)^{h \mathbb{Z}/p})$. Let $C'$ be
the cofiber of this map, in turn. It suffices to show that $C' \in
\mathcal{C}$. We see from the long exact sequence in
homotopy groups that $\pi_1(C') = 0$, $\pi_1(C'^{h\mathbb{Z}/p}) = 0$, and
$\pi_0(C')$ is torsion-free.  
However, we can now apply Lemma \ref{lotsofthingsmeanC'} to see that $C' \in \mathcal{C}$. 
\end{proof}

\pagestyle{plain}

\end{document}